\documentclass[11pt,reqno]{amsart}
\usepackage[a4paper,scale=0.75,twoside=false]{geometry}  
\usepackage[english]{babel}
\usepackage{microtype}
\usepackage{amsfonts}
\usepackage[utf8]{inputenc}	
\usepackage[T1]{fontenc}   
\usepackage{amsmath}
\usepackage{xfrac}
\usepackage{graphicx}
\usepackage[colorlinks=true, allcolors=black]{hyperref}
\usepackage{enumitem}
\usepackage{tikz}
\usepackage{amssymb}
\usepackage{amsthm}
\usepackage{color}
\usepackage{subcaption}
\usepackage{afterpage}
\usepackage{mathtools}
\usepackage{pdfpages}
\usepackage{bm}
\usepackage{appendix}
\usepackage{cleveref}
\usepackage{tikz}
\usepackage{comment}
\usepackage{float}
\usepackage{empheq}
\usepackage[nocompress]{cite}

\crefformat{equation}{(#2#1#3)}
\crefmultiformat{equation}{(#2#1#3)}%
{ and~(#2#1#3)}{, (#2#1#3)}{ and~(#2#1#3)}
\crefrangeformat{equation}{(#3#1#4) to~(#5#2#6)}

\newtheorem{theorem}{Theorem}[section]
\newtheorem{lemma}[theorem]{Lemma}
\newtheorem{proposition}[theorem]{Proposition}
\newtheorem{corollary}[theorem]{Corollary}

\theoremstyle{definition}
\newtheorem{definition}[theorem]{Definition}
\newtheorem{remark}[theorem]{Remark}

\AtBeginEnvironment{remark}{%
\pushQED{\qed}%
}
\AtEndEnvironment{remark}{\popQED\endremark}

\newcommand{\spann}{\mathrm{span}}

\newcommand{\R}{\mathbb{R}}
\newcommand{\N}{\mathbb{N}}
\newcommand{\Z}{\mathbb{Z}}
\newcommand{\T}{\mathbb{T}}
\newcommand{\J}{\mathcal{J}}
\newcommand{\M}{\mathcal{M}}

\title[Asymmetric waves for capillary Whitham equation]{Asymmetric travelling wave solutions of the capillary-gravity Whitham equation}

\author[O. Mæhlen]{Ola Mæhlen}
\email{olamaeh@math.uio.no}
\address{Department of Mathematics, University of Oslo, 0851 Oslo, Norway}

\author[D. S. Seth]{Douglas Svensson Seth}
\email{douglas.s.seth@ntnu.no}
\address{Department of Mathematical Sciences, Norwegian University of Science and Technology, 7491 Trondheim, Norway}


\begin{document}

\begin{abstract}
By a bifurcation argument we prove that the capillary-gravity Whitham equation features \textit{asymmetrical} periodic travelling wave solutions of arbitrarily small amplitude. Such waves exist only in the weak surface tension regime $0<T<\frac{1}{3}$ and are necessarily bimodal; they are located at double bifurcation points satisfying a certain symmetry breaking condition. Our bifurcation argument is an extension of the one applied by Ehrnström et al. in \cite{Ehrnstroem2019} to find symmetric waves: Here, two additional scalar equations arise. Combining the variational structure of our problem with its translation symmetry, we show that these two additional equations are in fact linearly dependent, and can (at `symmetry breaking' bifurcation points) be solved by incorporating the surface tension as a bifurcation parameter. Contrary to the symmetric case, only very specific modal pairs $(k_1,k_2)$ give rise to (small) asymmetrical periodic waves and we here provide a partial characterization of such pairs.
\end{abstract}
\maketitle
\section{Introduction}
The capillary-gravity Whitham equation, or capillary Whitham for short, is a one-dimensional shallow water wave model given by the equation
\begin{align}\label{eq: main}
    u_t + (M_Tu + u^2)_x=0,
\end{align}
where $u$ is the surface profile dependent on time $t\in\R$ and space $x\in\R$. Here $M_T$ is a spatial Fourier multiplier, characterized by
\begin{align*}
    \widehat{M_Tu}(t,\xi)= m_T(\xi)\hat{u}(t,\xi),
\end{align*}
where we use the Fourier transform $\widehat{u}(t,\xi)=\int_{\R}u(t,x)e^{-i\xi x}dx$ and where the symmetric symbol $m_T$ is given by
\begin{equation*}
 m_T(\xi)= \sqrt{\frac{(1+T\xi^2)\tanh{\xi}}{\xi}},
\end{equation*}
for some strictly positive surface tension parameter $T>0$. The symbol $m_T$ arises as the linearized dispersion relation for capillary-gravity shallow water waves as described by the Euler equations \cite{DLannes}. In the absence of surface tension, $T=0$, the model \cref{eq: main} reduces to the (purely gravitational) Whitham equation introduced in \cite{Whitham1967vma} as an alternative to the more strongly dispersive KdV equation. Over the last two decades the Whitham equation has been proven a particularly rich shallow water wave model: It admits smooth and peaked, periodic and solitary, travelling wave solutions \cite{EHRNSTROM20191603, solitaryWavesWhithamEhrnstormetal, travellingWaveWhithamEhrnstromKalisch, solitaryWavesTruong}, it features wave breaking in finite time \cite{WaveBreakingHur, WaveBreakSautWang}, and it models the Euler equations just as well as the KdV equation \cite{JCWhitham} (in some regimes it is even slightly better \cite{emerald:hal-03103819}). 

In the presence of surface tension, the natural generalisation of this water wave model is the capillary Whitham equation \cref{eq: main} introduced in \cite{MR3177639} whose modelling quality is further investigated in \cite{MR3628376}. The central difference between the two equations is the asymptotic behaviour of the symbol:
\begin{equation*}
m_T(\xi)\sim \begin{cases} \dfrac{1}{\sqrt{|\xi|}}, \hspace{16pt} T=0,\quad\text{(Whitham equation)}\vspace{6pt}\\ \sqrt{|T\xi|}, \quad T>0,\quad \text{(capillary Whitham equation)}\end{cases}
\end{equation*}
as $|\xi|\to\infty$; contrary to the Whitham equation, the capillary Whitham equation features dispersion of \textit{positive} order $\frac{1}{2}$. One implication is that all bounded travelling wave solutions of \cref{eq: main} are smooth (never peaked) which is demonstrated in \cite[Proposition 5.1]{Ehrnstroem2019}. It is expected \cite{JCWhitham} that this positive order dispersion also inhibits wave-breaking, while instead introducing a `bubble blow-up' phenomenon similar to what has been discovered for NLS \cite{MR1203233},  gKdV \cite{MinimalMassBlowupgKdV} and mBO \cite{YvanPilodBlowUp}. Such potential blow-up would rule out a global well-posedness theory for \cref{eq: main}, though said equation has been proved locally well-posed in $C(\R,H^s(\R))$ for $s>\frac{9}{8}$ \cite{MOLINET20181719}.

On the topic of travelling wave solutions of \cref{eq: main}, it is natural to rewrite the equation in its steady form: Inserting the travelling wave ansatz $(x,t)\mapsto u(x-ct)$ into \cref{eq: main} and integrating we obtain the steady capillary Whitham equation 
\begin{equation}\label{eq: steadyWhitham}
   (M_T-c)u+u^2=0,
\end{equation}
where $c>0$ is the wave speed and where we removed the constant of integration using a Galilean transformation argument. This equation features solitary wave solutions, both classical \cite{SolitaryWavesArnesen} and generalized  \cite{MR4057934}, but we will here focus solely on its \textit{periodic} solutions, of which the small and symmetric ones have been fully characterized \cite{Ehrnstroem2019} with properties and stability further examined in \cite{CHARALAMPIDIS2021102793, CarterRozman}. We here complete this work by characterizing the asymmetric ones as well; where we use the following `up to translation' definition of symmetry:
\begin{definition}[Symmetric and asymmetric functions]\label{def:asymmetric} \,\\
    We say that a function $u:\mathbb{R}\to \R$ is symmetric if there exists an $a\in\R$ such that $x\mapsto u(x+a)$ is an even function, and if no such $a$ exists then we say that $u$ is asymmetric.
\end{definition}
\noindent As we shall see, asymmetric periodic solutions of \cref{eq: steadyWhitham} are far more scarce than symmetric ones. Moreover, small periodic travelling wave solutions are necessarily symmetric for strong surface tension $T\geq \frac{1}{3}$ by Corollary \ref{cor: simpleBifurcationPointsGiveRiseToOnlySymmetricSolutions}. 

The scarcity of asymmetric periodic travelling waves is a peculiarity of the full Euler equations as well; in fact, there are several results guaranteeing the symmetry of a wave in specific settings \cite{GarabedianSymmetry, symmetryOfSolitaryWaves,ConstantinSymmetry,MatsSymmetry}. Further complicating the matter, potential branches of asymmetric waves bifurcating from the trivial flow (zero) have been ruled out both in the gravity \cite{JMPA_1934_9_13__217_0} and capillary-gravity setting \cite{NonExistenceOfAsymmetricBifurcationPointsGravityCapillaryOkaSho} for fixed surface tension. This leaves open two options: A branch of asymmetric travelling waves can connect to a branch of symmetric ones (away from zero) at so-called `symmetry-breaking points'. Such points have been found, thus yielding asymmetric travelling water waves, both gravity and capillary-gravity, on finite and infinite depths \cite{zufiria_1987model,zufiria_1987gravity,zufiria_1987capillary,Shimizuetal}. Alternatively, an asymmetric branch may never connect to a symmetric one; loops of asymmetric waves have indeed been found  numerically \cite{gao_wang_vanden-broeck_2017} in the capillary-gravity finite depth setting. Of relevance is also the recent discovery of a large class of (potentially non-periodic) asymmetric waves \cite{Kozlov_2019}. Finally, we mention that the Kawahara equation, which serves as a water-wave model in the regime $(c,T)\sim (1,\frac{1}{3})$ \cite{Buffoni_Groves_1996}, has also been proved to feature asymmetrical periodic travelling wave solutions \cite{Buffoni_Champneys_1996}.

In this paper, we give an analytic proof for the existence of asymmetric periodic solutions of \cref{eq: steadyWhitham} bifurcating from the trivial flow $u=0$.\footnote{Equation \cref{eq: steadyWhitham} admits a second trivial flow $u=c-1$. But this scenario is equivalent to the considered one due to the Galilean symmetry $(u,c)\mapsto (u+c-1, 2-c)$ of the equation.} As demonstrated in \cite{Ehrnstroem2019}, nontrivial periodic solutions $u$ of \cref{eq: steadyWhitham} can be constructed near bifurcation points (Definition \ref{def: simpleAndDoubleBifrucationPoints}) using a solution $v$ of the linearized equation and fixed point arguments. Crucially, Proposition \ref{prop: UIsSymmetricIfAndOnlyIfVIs} tells us that $u$ will be asymmetric if, and only if, $v$ is. Thus, asymmetric solutions are found near bifurcation points where
\begin{enumerate}[label=(\roman*)]
    \item the linearization of \cref{eq: steadyWhitham} admits asymmetric solutions $v$,
    \item the construction of $u$ from $v$ is possible.
\end{enumerate}
The first step is only satisfied at double bifurcation points, and these exist exclusively for weak surface tensions $0<T<\frac{1}{3}$ (see Corollary \ref{cor: simpleBifurcationPointsGiveRiseToOnlySymmetricSolutions}). As for the second step, this turns out to be more challenging than in the symmetric case:
In \cite{Ehrnstroem2019} symmetry of $v$ ensures that, after a Lyapunov--Schmidt reduction, step (ii) consists of solving two finite dimensional equations, and they are dealt with by making appropriate changes to the wave's speed and period. 

When lacking symmetry of $v$, step (ii) reduces instead to solving four equations \cref{eq: decomposingVIntoCosinesAndSines}. A key result here is Corollary \ref{corollary: JOfUIsOrthogonalToVPrime} which exploits the translation symmetry of \cref{eq: steadyWhitham} to establish the linear dependence of two of these equations.\footnote{This trick aligns with the broader theory of \cite{Golubitsky1985} on the Lyapunov--Schmidt reduction argument applied to equations that commute with a compact group of symmetries.} Thus, only three need solving which, under special circumstances, can be done by tweaking the wave's speed, period \textit{and} surface tension. These `special circumstances' refer to the symmetry breaking criterion of Definition \ref{def: symmetryBreakingBifurcationPoints}, a property that only depends on the particular bifurcation point. Thus, at bifurcation points where the criterion is satisfied, one can construct an asymmetric solution $u$ from every choice of (small and) asymmetric $v$ giving then rise to a four-dimensional manifold of asymmetrical solutions (Theorem \ref{thm: 4DimensionalManifoldOfSolutionsFromNondegenerateSymmetryBreakingBifurcationPoint}).

However, the symmetry breaking criterion is rarely satisfied: Contrary to the results of \cite{Ehrnstroem2019}, where symmetric bi-modal waves of any two wave numbers $k_1$ and $k_2$ can be constructed, only some wave number pairs allow for symmetry breaking. We provide a necessary condition for such pairs (Theorem \ref{thm: necessaryConditionOnWaveNumberPairs}) and prove analytically that $(k_1,k_2)=(2,5)$ is the first one (Theorem \ref{thm: 25admitsSymmetryBreaking}). Numerical results suggests that there are an infinitude of wave number pairs giving rise to asymmetric waves (see Figure \ref{fig:k1k2pairs}), but we have not been able to prove this conjecture.  

We conclude the introduction with a coarse summary of our main result which are presented in greater detail in section \ref{sec: mainResults}.

\begin{theorem}[Summary of main result]\label{thm: summaryOfMainResults}
    There exist asymmetric periodic travelling wave solutions to the Whitham equation with arbitrarily small amplitude. These are found only near double bifurcation points that satisfy a symmetry breaking condition. The wave number pair $(k_1,k_2)=(2,5)$ is the smallest one that admits a corresponding `symmetry breaking' bifurcation point.  
\end{theorem}

\section{Setup and main results}\label{sec: SetupAndMainResults}
This section explains our approach to finding asymmetric periodic solutions of \cref{eq: steadyWhitham} and aims to motivate some of the technical calculations in Section \ref{sec: preliminaries}, \ref{sec: LocalBifurcationTheoryToSymmetryBreakingPoints} and \ref{sec: symmetryBreakingWavePairs}.

\subsection{Formulating the problem}
We begin with a proper formulation of the problem:  In seeking periodic solutions of \cref{eq: steadyWhitham} we may restrict our attention to the $2\pi$-setting by instead introducing a scaling parameter in $M_T$: If $v$ is $2P$-periodic and solves \cref{eq: steadyWhitham}, then $u(\cdot)\coloneqq v(\cdot/\kappa )$ with $\kappa = \pi/P$, is $2\pi$-periodic and solves 
\begin{equation}\label{eq: scaledSteadyCapillaryWhithamEquation}
     J(u; c,\kappa, T)\coloneqq (M_{T,\kappa}-c)u+u^2=0,
\end{equation}
where $ M_{T,\kappa}$ is the Fourier multiplier with symbol
	\[ 
	\xi\mapsto m_T(\kappa\xi)=\sqrt{\frac{(1+T\kappa^2\xi^2)\tanh{\kappa\xi}}{\kappa\xi}}.
	\]
We will furthermore consider the value of the surface tension $T>0$ as part of the solution as the coming bifurcation arguments will require us to allow modifications to $T$. In summary, the periodic travelling waves of \cref{eq: main} can all be represented (up to Galilean symmetry) by quadruples   $(u,c,\kappa,T)\in L^\infty(\mathbb{T})\times \R^3_+$ where $\mathbb{T}=\mathbb{R}/2\pi\mathbb{Z}$. Such representations are not unique as the pair $(u,\kappa)$ can be replaced by $(u(n\cdot ),n\kappa)$ for any $n\in\N$ yielding the same solution of \cref{eq: steadyWhitham}; this symmetry plays little importance to us, but justifies the coprime assumption we pose on double bifurcation points in Remark \ref{rem: oneMayAssumeThatK1AndK2AreRelativePrimes}.

In light of Proposition \ref{prop: solutionsAreSmooth} (and Remark \ref{rem: searchingForHighLowRegularitySolutionsAreEquivalentTasks}) there is no harm in restricting our search of solutions to a Zygmund space $\mathcal{C}^s(\T)$ \cite{Taylor_PDEs3} with $s>1$. This choice of smoothness is for technical convenience as both $u'$ and $M_{T,\kappa}u$ are then well defined continuous functions for any $u\in \mathcal{C}^s(\T)$; here we tacitly used that $M_{T,\kappa}$ is a bijection between $\mathcal{C}^s(\T)$ and $\mathcal{C}^{s-\sfrac{1}{2}}(\T)$ (see \cite{Ehrnstroem2019}).

\subsection{Characterizing the bifurcation points}\label{sec: characterizingBifurcationPoints}
Viewing $u\mapsto J(u)$ as an operator from $\mathcal{C}^s(\T)$ to $\mathcal{C}^{s-\sfrac{1}{2}}(\T)$ its Fréchet derivative at $u=0$ is the operator
\begin{align}\label{eq: theLinearizationOperator}
    D_uJ(0;c,\kappa,T) =  M_{T,\kappa}-c.
\end{align}
By Proposition \ref{prop: theLinearizationIsFredholm}, this operator is Fredholm of index zero for every choice of $c,\kappa,T>0$, with an analytic dependence on these parameters. Thus, by a Lyapunov--Schmidt reduction we can construct (nontrivial) solutions from the trivial flow $u=0$ provided that $\ker M_{T,\kappa}-c \neq \{0\}$ which, by the same proposition, is true if, and only if, $c=m_T(\kappa k)$ for some $k\in \N_0$. Whenever this condition is met, we refer to $(0,c,\kappa,T)$, or simply the triplet $(c,\kappa,T)$, as a \textit{bifurcation} point. The case $c=1$ (and $k=0$), will be of no interest to us as it produces no asymmetric (periodic) solutions; see Remark \ref{remark: specialBifurcationForCEqualToOne}. By Lemma \ref{lemma:mTproperties}, this leaves us with two types of bifurcation points:

\begin{definition}[Simple and double bifurcation points]\label{def: simpleAndDoubleBifrucationPoints}
    We say that a bifurcation point $(c_0,\kappa_0,T_0)\in \R_+^3$ is \textit{simple} if there is precisely one $k\in\Z_+$ such that $c_0=m_{T_0}(\kappa_0 k)$ meaning that
      \begin{equation*}
        \ker D_uJ(0;c_0,\kappa_0,T_0) = \mathrm{span}\{\cos(kx),\sin(kx)\},
    \end{equation*}
    and we say that it is \textit{double} if there are exactly two distinct $k_1,k_2\in\Z_+$ such that $c_0=m_{T_0}(\kappa_0 k_1)=m_{T_0}(\kappa_0 k_2)$ meaning that
    \begin{equation*}
        \ker D_uJ(0;c_0,\kappa_0,T_0) = \mathrm{span}\{\cos(k_1x),\cos(k_2x),\sin(k_1x),\sin(k_2x)\}.
    \end{equation*}
\end{definition}

\begin{remark}\label{rem: oneMayAssumeThatK1AndK2AreRelativePrimes}
    When working with a simple bifurcation point $(c_0,\kappa_0,T_0)$, there is no harm in assuming that the corresponding wave number $k$ is equal to one. For if not, then $(c_0,k\kappa_0,T_0)$ would be a simple bifurcation point with said property and, moreover, these two points would yield the exact same solutions of the \textit{unscaled} equation \cref{eq: steadyWhitham}. 
    
   Similarly, when working with a double bifurcation we can, by an analogous scaling argument, safely assume that $k_1$ and $k_2$ are coprime, i.e.~that their greatest common divisor is $1$.
\end{remark}
For strong surface tensions $T\geq \frac{1}{3}$ every bifurcation point is simple (Corollary \ref{cor: allBifurcationPointsAreSimpleInTheStrongSurfaceTensionRegime}). But simple bifurcation points are of no interest to us; they locally only give rise to symmetric solutions (Corollary \ref{cor: simpleBifurcationPointsGiveRiseToOnlySymmetricSolutions}). Consequently, we will focus on weak surface tensions $T\in(0,\frac{1}{3})$ where double bifurcation points are abundant. In fact, for any $T\in (0,\frac{1}{3})$ and any wave number pair $(k_1,k_2)$ there exists unique corresponding bifurcation point $(c,\kappa,T)$; this summed up in Proposition \ref{prop: c0AndKappa0AreDeterminedByK1K2AndT}. 

For the remainder of the section we assume a fixed choice of wave numbers $1\leq k_1<k_2$ satisfying $\gcd(k_1,k_2)=1$ (c.f.~Remark \ref{rem: oneMayAssumeThatK1AndK2AreRelativePrimes}) with some corresponding double bifurcation point $(c_0,\kappa_0,T_0)$. Again, by Proposition \ref{prop: c0AndKappa0AreDeterminedByK1K2AndT}, $c_0,\kappa_0$ are uniquely determined by $k_1,k_2$ and the (seemingly) free parameter $T_0\in (0,\frac{1}{3})$; we shall shortly demonstrate that only certain very specific values of $T_0$ give rise to asymmetric solutions. 

\subsection{Lyapunov--Schmidt reduction}\label{sec: lyapunovSchmidtArgument}

For brevity, we introduce the notation
\begin{equation}\label{eq: theSpaces}
\begin{split}
V\coloneqq \spann\{\cos(k_1x),\cos(k_2x),\sin(k_1x),\sin(k_2x)\},\hspace{10pt}\\
   W\coloneqq V^\perp,\hspace{30pt}   W^s\coloneqq \mathcal{C}^s(\T)\cap V^{\perp}, \hspace{30pt}U^s\coloneqq \mathcal{C}^s(\T),
\end{split}
\end{equation}
that is, $V$ is the kernel of $DJ_u(0) = M_{T_0,\kappa_0}-c_0$ and $W$ is its orthogonal complement in $L^2(\T)$. 
Let further $P_V$ and $P_W$ be the $L^2$-orthogonal projections onto $V$ and $W$ respectively; we can then decompose any element $u\in U^s$ into $v\in V$ and $w\in W^s$ by 
\begin{equation*}
    u = P_Vu + P_Wu \eqqcolon v + w.
\end{equation*}
Decomposing similarly equation \cref{eq: scaledSteadyCapillaryWhithamEquation} we get the (equivalent) the system
\begin{subequations}\label{eq: rewritingTheMainEquationIntoTwoEquations}
    \begin{empheq}[left={\empheqlbrace\,}]{align}
      & P_VJ(u)=(M_{T,\kappa}-c)v + P_V(v+w)^2=0,
        \label{eq: equationForV} \\
      & P_WJ(u)=(M_{T,\kappa}-c)w + P_W(v+w)^2=0,
        \label{eq: equationForW}
    \end{empheq}
\end{subequations}
where we used that the projections $P_V,P_W$ commute with $(M_{T,\kappa}-c)$ as they are all Fourier multipliers. The aim is to obtain solutions of the system \cref{eq: rewritingTheMainEquationIntoTwoEquations}
by fixing $v$ and then modify $c,\kappa,T$ and $w$ appropriately; the resulting solution $u$ will be asymmetric if, and only if, $v$ is (Proposition \ref{prop: UIsSymmetricIfAndOnlyIfVIs}). 

Conveniently, \cref{eq: equationForW} can be uniquely solved for $w$ given any quadruplet $(v,c,\kappa,T)$ sufficiently close to the double bifurcation point $(0,c_0,\kappa_0,T_0)$ (Proposition \ref{prop: infdimAKAuniqueWSolution}). Inserting then for $w$ in \cref{eq: equationForV} the problem reduces to the finite dimensional system \cref{eq: equationForV}, which consists of four scalar equations (as $V$ is four-dimensional). Here a subtle issue arises: We have merely three free parameters $c,\kappa,T$, and so it may seem that \cref{eq: equationForV} is overdetermined. This is surprisingly not the case as the equation is trivially satisfied in the $v'$ direction for any $v,c,\kappa,T$. Indeed, Corollary \ref{corollary: JOfUIsOrthogonalToVPrime} tell us that
\begin{equation}\label{eq: triviallySolvedInVPrimeIntroVersion}
    \langle J(u),v'\rangle=0
\end{equation}
 whenever \cref{eq: equationForW} holds and, thus, we only need to find $c,\kappa,T$ such that \cref{eq: equationForV} holds in the three remaining directions orthogonal to $v'$. Still, solving \cref{eq: equationForV} is not easy, and hinges on the symmetry breaking criterion introduced below.

\subsection{The symmetry breaking criterion}\label{sec: theSymmetryBreakingCondition}
To continue in a meaningful manner, we will need a convenient representation for $v\in V$. By the sum formula for cosine, we can pick $r_1,r_2,\theta_1,\theta_2\in\R$ such that $v$ may be written
\begin{equation}\label{eq: cosineRepresentationOfV(x)}
    v(x)=r_1\cos\Big(k_1(x+\theta_1)\Big)+r_2\cos\Big(k_2(x+\theta_2)\Big).
\end{equation}
Such parameters are not unique, but any choice will do for this discussion. We can then write \cref{eq: equationForV} as a system of the four scalar equations
\begin{subequations}
  \label{eq: decomposingVIntoCosinesAndSines}
    \begin{empheq}{align}
     \Big\langle J(u),\cos\big(k_1(x+\theta_1)\big)\Big\rangle = 0,\quad \Big\langle J(u),\cos\big(k_2(x+\theta_2)\big)\Big\rangle=0,\label{eq: theCosineSpace}\\
    \Big\langle J(u),\sin\big(k_1(x+\theta_1)\big)\Big\rangle = 0,\quad\Big\langle J(u),\sin\big(k_2(x+\theta_2)\big)\Big\rangle=0.
        \label{eq: theSineSpace}
    \end{empheq}
\end{subequations}
The two top equations \cref{eq: theCosineSpace} can be uniquely solved in $c$ and $\kappa$ (Proposition \ref{prop:findimpart1}), while the two lower equations \cref{eq: theSineSpace} are linearly dependent by \cref{eq: triviallySolvedInVPrimeIntroVersion} and so we need only solve one of them. Inserting then for $w,c$ and $\kappa$ in the left-most equation in \cref{eq: theSineSpace} we get
\begin{equation}\label{eq: theDifficultEquation}
     \Big\langle \big(J\big(v(x)+w(x;v,T)\big),\sin\big(k_1(x+\theta_1)\big)\Big\rangle=0.
\end{equation}
where we have emphasized $w$'s implicit dependence on $v$ and $T$. To solve \cref{eq: theDifficultEquation}, in the last free parameter $T$, is the true difficulty of finding asymmetrical solutions. Proposition \ref{proposition:lastfineq} offers the following factorization of the left-hand side so that \cref{eq: theDifficultEquation} can be written
\begin{align}\label{eq: rewritingTheLeftHandSideOfTheDifficultEquation}
    \Big[\phi(T_0; k_1,k_2) + o(1)\Big]\Big[r_1^{k_2}r_2^{k_1}\sin\big(k_1k_2(\theta_1-\theta_2)\big)\Big]=0,
\end{align}
where the small-$o$ term vanishes as $(v,T)\to(0,T_0)$ and where $\phi$ is an analytic function for $T\in(0,\frac{1}{3})$ properly defined in Definition \ref{def:phidef}. By Proposition \ref{proposition: criterionForAsymmetryOfV}, the latter square bracket in \cref{eq: rewritingTheLeftHandSideOfTheDifficultEquation} is zero if, and only if, $v$ is symmetric.\footnote{This is one way of seeing that the two equations \cref{eq: theSineSpace} are trivially satisfied when $v$ is symmetric.} Thus, \cref{eq: rewritingTheLeftHandSideOfTheDifficultEquation} can only hold for asymmetric $v$ if $\phi(T_0;k_1,k_2)=0$ (at least when $(v,T)$ is sufficiently close to $(0,T_0)$). This is in essence the symmetry breaking criterion, which determines the value of $T_0$ in terms of $k_1,k_2$. The below definition asks for a little more so to control the small-$o$ term in \cref{eq: rewritingTheLeftHandSideOfTheDifficultEquation}.

\begin{definition}[Symmetry breaking criterion]\label{def: symmetryBreakingBifurcationPoints}
    Let $(c_0,\kappa_0,T_0)$ be a double bifurcation point and let $k_1,k_2\in \Z_+$ denote the corresponding wave number pair. With $\phi$ as in Definition \ref{def:phidef}, we say that $(c_0,\kappa_0,T_0)$ is a \textit{symmetry breaking bifurcation point} if $\phi(T_0;k_1,k_2)=0$ and $\phi$ is strictly monotone in some neighborhood of $T_0$. If additionally $\partial_T\phi(T_0;k_1,k_2)\neq 0$ we say the point is \textit{nondegenerate} and otherwise \textit{degenerate}.
    
    Similarly, we will say that a wave number pair $(k_1,k_2)$ \textit{admits symmetry breaking} if there exists at least one symmetry breaking bifurcation point for that wave number pair.
\end{definition}
\begin{remark}\label{rem:phiT1T2oppositesign} The analyticity of $T\mapsto \phi(T;k_1,k_2)$ ensures that $(k_1,k_2)$ admits symmetry breaking if, and only if, there are $T_1,T_2\in(0,1/3)$ such that $\phi(T_1;k_1,k_2)<0<\phi(T_2;k_1,k_2)$.
 \end{remark}
\begin{remark}\label{rem: allSymmetryBifurcationPointsAreProbablyNondegenerate}
    While the previous definition is quite general, we actually expect all symmetry breaking bifurcation points to be nondegenerate: The countable family $\{\phi(\cdot,k_1,k_2)\}_{k_1,k_2\in \N}$ consists of somewhat generic analytic functions on $(0,\frac{1}{3})$ and so, heuristically speaking, it is unlikely that one ever encounters the scenario $\phi(T_0;k_1,k_2)=\partial_T\phi(T_0;k_1,k_2)=0$.
   But we lack efficient ways of analysing this intricately defined family, and so we can not rigorously justify the conjecture. 
\end{remark}

\subsection{Main results}\label{sec: mainResults}
We begin by presenting the results of Section \ref{sec: LocalBifurcationTheoryToSymmetryBreakingPoints}, where we rigorously carry out the above described steps and show that symmetry breaking bifurcation points give rise to a structure of both symmetric and asymmetric solutions of \cref{eq: scaledSteadyCapillaryWhithamEquation}; this is Theorem \ref{thm:main}. The topology of this structure can be characterized in the nondegenerate case, which we expect is the typical case (Remark \ref{rem: allSymmetryBifurcationPointsAreProbablyNondegenerate}), and is the one we present here:

\begin{theorem}\label{thm: 4DimensionalManifoldOfSolutionsFromNondegenerateSymmetryBreakingBifurcationPoint}
    Let $(0,c_0,\kappa_0,T_0)\in \mathcal{C}^s(\T)\times \R_+^3$, with $s> 0$, be a nondegenerate symmetry breaking bifurcation point of \cref{eq: scaledSteadyCapillaryWhithamEquation}, and let $k_1,k_2\in \Z_+$ be the corresponding wave numbers. Then, for some $\varepsilon>0$, there is a four-dimensional manifold $\mathcal{M}$ of solutions of \cref{eq: scaledSteadyCapillaryWhithamEquation} in $\mathcal{C}^s(\T)\times \R_+^2\times(0,\frac{1}{3})$,
    \begin{align*}
        \mathcal{M}=\Big\{\big(u(r,\theta),c(r,\theta),\kappa(r,\theta),T(r,\theta)\big) : (r,\theta)\in [0,\varepsilon)^2\times\Theta\Big\},
    \end{align*}
    parameterized over $r=(r_1,r_2)\in [0,\varepsilon)^2$ and $\theta=(\theta_1,\theta_2)\in \Theta \coloneqq \big(\R/\tfrac{2\pi}{k_1}\Z\big)\times \big(\R/\tfrac{2\pi}{k_2}\Z\big)$. All four mappings are analytical in $(r,\theta)$ and of leading orders
    \begin{equation}\label{eq: leadingOrderPartOfTheSolutionsFromTheManifold}
    \begin{split}
          u(x; r,\theta) =&\, r_1\cos\big(k_1(x+\theta_1)\big) +r_2\cos\big(k_2( x+\theta_2)\big)+ \mathcal{O}(r^2),\\[6pt]
          c(r,\theta)=c_0 +&\, \mathcal{O}(r^2),\quad \kappa(r,\theta)=\kappa_0 + \mathcal{O}(r^2),\quad T(r,\theta)=T_0 + \mathcal{O}(r^2).
    \end{split}
    \end{equation}
    The function $x\mapsto u(x;r,\theta)$ is asymmetric if, and only if, $r_1r_2\neq 0$ and $\theta_1-\theta_2\notin \frac{\pi}{k_1k_2}\Z$. Finally, $\mathcal{M}$ includes all asymmetric solutions of \cref{eq: scaledSteadyCapillaryWhithamEquation} that are sufficiently close to $(0,c_0,\kappa_0,T_0)$. 
\end{theorem}
\begin{remark}
    We make three remarks regarding the manifold $\mathcal{M}$:
    \begin{enumerate}
        \item By a more detailed argument than the one presented in \Cref{sec: LocalBifurcationTheoryToSymmetryBreakingPoints}, one can show that $\mathcal{M}$ is homeomorphic to $\R^4$ (as opposed to $[0,\varepsilon)^2\times\Theta$); in fact, any point on $\mathcal{M}$ is uniquely determined by the leading order part of $u(x;r,\theta)$ in \eqref{eq: leadingOrderPartOfTheSolutionsFromTheManifold}.
        \item The symmetric elements of $\M$ constitutes a three-dimensional closed structure dividing $\M$ into two components (each homeomorphic to $\R^3\times\T$) of asymmetric elements. See Remark \ref{rem: describingTheTopologyOfV} for more details. 
        \item $\mathcal{M}$ does not include all \textit{symmetric} solutions of \eqref{eq: scaledSteadyCapillaryWhithamEquation} in the vicinity of $(0,c_0,\kappa_0,T_0)$. Three additional four-dimensional manifolds of symmetric solutions can be constructed. See \Cref{remark: aDescriptionOfAllSymmetricSolutions} for further details. \qedhere
    \end{enumerate}
\end{remark}
Our second set of findings pertains to Section \ref{sec: symmetryBreakingWavePairs} where we investigate which wave number pairs $(k_1,k_2)$ that allow for symmetry breaking (c.f.~Definition \ref{def: symmetryBreakingBifurcationPoints}). A full characterization of such pairs seems out of reach due to their dependence on the implicit functions $\{\phi(\cdot,k_1,k_2)\}_{k_1,k_2\in\N}$ from Definition \ref{def:phidef}. Still, we are able to give the necessary condition that $k_1\nmid k_2$ and $(k_2-k_1)\nmid k_1$, which is proved by Proposition \ref{prop:NoSymmetryBreaking} and summed up by the following theorem:
\begin{theorem}\label{thm: necessaryConditionOnWaveNumberPairs}
    Let $1\leq k_1<k_2$ be integers. Then, there exists no symmetry breaking bifurcation point with corresponding wave number pair $(k_1,k_2)$ if either $k_1$ divides $k_2$ or $k_2-k_1$ divides $k_1$.
\end{theorem}
When searching for wave number pairs that allow for symmetry breaking, one finds that $(2,5)$ arises as the smallest candidate not excluded by the previous theorem. This pair is small enough so that $\phi(T;2,5)$ can be computed semi-explicitly as done in Proposition \ref{prop:intersection25}. And indeed, we find that this pair is the smallest that admits symmetry breaking:
\begin{theorem}\label{thm: 25admitsSymmetryBreaking}
    The wave number pair $(2,5)$ admits symmetry breaking in accordance with Definition \ref{def: symmetryBreakingBifurcationPoints}. 
\end{theorem}
The theorem is proved by Proposition \ref{prop:intersection25} whose proof does not easily generalize to larger pairs due to the impracticality of manually computing $\phi(T; k_1,k_2)$ (see Remark \ref{rem: ExactExpressionForNAndM}). Still, numerical evidence indicates that there are plenty more pairs admitting symmetry breaking as seen in Figure \ref{fig:k1k2pairs}.

We end the section by explaining how Theorem \ref{thm: summaryOfMainResults} follows:
Applying the more general Theorem \ref{thm:main} to the implied symmetry breaking bifurcation point of Theorem \ref{thm: 25admitsSymmetryBreaking}, we get asymmetric travelling wave solutions for the capillary Whitham equation. In fact, Figure \ref{fig:25Tplot}, which plots $T\mapsto \phi(T;2,5)$ on $(0,\frac{1}{3})$, indicates that this bifurcation point is unique and, more importantly, nondegenerate so that Theorem \ref{thm: 4DimensionalManifoldOfSolutionsFromNondegenerateSymmetryBreakingBifurcationPoint} applies. The last part of the theorem is covered by \Cref{prop:phizeronecessary}.

\section{Preliminaries}\label{sec: preliminaries}
This section collects some auxiliary results and observations on which the arguments in Section \ref{sec: LocalBifurcationTheoryToSymmetryBreakingPoints} hinges on.

\subsection{Zygmund spaces}
We here give a definition of Zygmund spaces on the circle. These spaces are natural for our setting ($M_T$ is an isomorphism between $\mathcal{C}^{s+\frac{1}{2}}(\T)$ and $\mathcal{C}^{s}(\T)$ for all $s$) though their exact structure will matter little to us; we refer the reader to \cite[Section 13.8]{Taylor_PDEs3} for a thorough introduction to this subject. 
Let $\varphi$ be defined as the Fourier inverse of a symmetric real function $\hat{\varphi}\in C_c^\infty(\R)$ which satisfies $\hat{\varphi}(\xi)=1$ for $|\xi|\leq 1$ and $\hat{\varphi}(\xi)=0$ for $|\xi|\geq 2$. We then set $\phi_0\coloneqq \varphi$ and $\phi_j(x)\coloneqq 2^j\varphi(2^jx) - 2^{j-1}\varphi(2^{j-1}x)$ for $j\in \N$ and define $\mathcal{C}^{s}(\T)$ as the set of real-valued distributions $u$ on $\T=\R/2\pi\Z$ such that
\begin{align*}
    \|u\|_{\mathcal{C}^{s}(\T)}\coloneqq \sup_{j} 2^{js}\|u\ast \phi_j\|_{L^\infty(\T)}<\infty.
\end{align*}
Up to a re-scaling of the norm, the space $\mathcal{C}^{s}(\T)$ is a Banach algebra for each $s\geq 0$ and for all $s'>s\geq 0$ the space $\mathcal{C}^{s'}(\T)$ embeds compactly in $\mathcal{C}^s(\T)$. 
Moreover, the Hölder space $C^s(\T)$ naturally embeds in $\mathcal{C}^s(\T)$ for all $s\geq 0$, while the converse is true only for non-integer $s$.

In particular, every element of $\mathcal{C}^s(\T)$ is a continuous function when $s>0$.

\subsection{Smoothing effects and Fredholm property of the linear part}
The following result is a slightly stronger variant of \cite[Proposition 5.1]{Ehrnstroem2019}. 

\begin{proposition}[Solutions are smooth]\label{prop: solutionsAreSmooth}
    Let $u\in L^\infty(\T)$ and $c,\kappa,T>0$. If $(u,c,\kappa,T)$ solves \cref{eq: scaledSteadyCapillaryWhithamEquation} in distributional sense, then $u$ is smooth. In particular, for any sufficiently large bound $B>0$ we can find a corresponding constant $C_B$ such that 
    \begin{equation}\label{eq ZygmundNormBoundedByHeight}
        \|u\|_{\mathcal{C}^{s}(\T)}\leq C_B \|u\|_{L^{\infty}(\T)},
    \end{equation}
whenever $\max\{s,\|u\|_{L^{\infty}(\T)},c,\kappa^{-1},T^{-1}\}< B$.
\end{proposition}
\begin{proof}
    We only sketch the steps as this result will not matter to our analysis, and is rather meant to justify our smoothness assumption of $u$. The first part of the proposition is \cite[Proposition 5.1]{Ehrnstroem2019}, and by carrying out the bootstrap argument in the proof of said proposition, one also gets the stated inequality. 
\end{proof}
\begin{remark}\label{rem: searchingForHighLowRegularitySolutionsAreEquivalentTasks}
      Inequality \cref{eq ZygmundNormBoundedByHeight} guarantees that solutions found sufficiently close to zero in $L^\infty(\T)$ are in-fact close to zero in $\mathcal{C}^{s}(\T)$ as well, so that we can safely work in a high-regularity space without potentially overlooking any solutions `near' zero.
\end{remark}

The following result is essential to the Lyapunov--Schmidt reduction argument.
\begin{proposition}[Fredholm property of the linearized equation]\label{prop: theLinearizationIsFredholm}
    Let $s>1$, and  $T,c,\kappa>0$. Then the operator $(M_{T,\kappa}-c)\colon\mathcal{C}^{s}(\T)\to\mathcal{C}^{s-\frac{1}{2}}(\T)$ is Fredholm of index zero. Moreover, we have the characterization
    \begin{equation*}
        \ker \big(M_{T,\kappa}-c\big) = \spann\big\{\cos(kx),\sin(kx):\, \text{for $k\in \N_0$ such that } m_{T}(\kappa k)=c\big\}.
    \end{equation*}
\end{proposition}
\begin{proof}
By \cite[Lemma 2.4]{Ehrnstroem2019}, $M_{T,\kappa}^{-1}$ is a bijection between $\mathcal{C}^{s-\frac{1}{2}}(\T)$ and $\mathcal{C}^{s}(\T)$, and since $\mathcal{C}^{s}(\T)$ embeds compactly in $\mathcal{C}^{s-\frac{1}{2}}(\T)$ it follows that $cM_{T,\kappa}^{-1}$ is a compact operator on $\mathcal{C}^{s-\frac{1}{2}}(\T)$. Thus $1-cM^{-1}_{T,\kappa}$ is a Fredholm operator of index zero on $\mathcal{C}^{s-\frac{1}{2}}(\T)$ (as it is a compact pertubation of the identity).  We conclude that 
\begin{equation*}
    M_{T,\kappa}-c = (1-cM_{T,\kappa}^{-1})M_{T,\kappa}
\end{equation*}
is a Fredholm operator of index zero from $\mathcal{C}^{s}(\T)$ to $\mathcal{C}^{s-\frac{1}{2}}(\T)$. Finally, the expression for the kernel can be seen as follows:
Since $s-\frac{1}{2}>\frac{1}{2}$, it follows that elements of $\mathcal{C}^{s-\frac{1}{2}}(\T)$ have convergent Fourier series \cite{katznelson_2004}. In particular, writing $u\in \mathcal{C}^{s}(\T)$ in terms of its Fourier series we find 
\begin{align*}
    u(x) =&\,  a_0 + \sum_{k=1}^\infty a_k\cos(kx)+b_k\sin(kx),\\
   \implies \big(M_{T,\kappa}-c\big)u (x) =&\,(1-c)a_0 + \sum_{k=1}^\infty \big(m_{T}(\kappa k)-c\big)\big(a_k\cos(kx)+b_k\sin(kx)\big),
\end{align*}
and so we conclude that the kernel of $M_{T,\kappa}-c$ is as described.
\end{proof}

\subsection{Properties of the Fourier symbol}
The following lemma is proved in the appendix of \cite{TRICHTCHENKO2016147} and characterizes the behaviour of the symbol $m_T$ in the three regimes of zero ($T=0$), weak ($0<T<\frac{1}{3})$, and strong ($T\geq \frac{1}{3}$) surface tension.
\begin{lemma}[Behaviour of $m_T$]\label{lemma:mTproperties}
    If $T=0$ then $m_T'(\xi)< 0$ for all $\xi>0$. If instead $T\geq \frac{1}{3}$ then $m_T'(\xi)> 0$ for all $\xi>0$. Finally, if $T\in(0,\frac{1}{3})$ then there is a number $\xi_T\in(0,\infty)$ such that 
    \begin{align*}
        m_T'(\xi)<&\,0 \text{ for } \xi\in(0,\xi_T),\\
            m_T'(\xi)>&\,0 \text{ for } \xi\in(\xi_T,\infty).
    \end{align*}
Moreover,  $T\mapsto \xi_T$ is analytic and strictly decreasing on $(0,\frac{1}{3})$, and it admits the limits $\lim_{T\downarrow 0}\xi_T=\infty$ and $\lim_{T\uparrow \frac{1}{3}}\xi_T=0$.
\end{lemma}
\begin{corollary}\label{cor: allBifurcationPointsAreSimpleInTheStrongSurfaceTensionRegime}
    All bifurcation points $(c_0,\kappa_0,T_0)$ with either $T_0\geq \frac{1}{3}$ or $T_0=0$ are simple.
\end{corollary}
\begin{proof}
    By the previous lemma we find that $k\mapsto m_{T_0}(\kappa_0 k)$ is strictly monotone on $\N_0$ and so there can be at most one solution to the equation $m_{T_0}(\kappa_0 k)=c_0$ for $k\in \N_0$.
\end{proof}
The previous corollary shows that double bifurcation points can only be found for (weak) surface tensions in the interval $(0,\frac{1}{3})$. In this regime, such points admit the following characterization:
\begin{proposition}\label{prop: c0AndKappa0AreDeterminedByK1K2AndT}
    For any $T\in(0,\frac{1}{3})$ and two integers $1\leq k_1<k_2$ there exists a unique pair of positive numbers $(c_0,\kappa_0)\in \mathbb{R}_+^2$ such that
    \begin{align*}
        m_T(\kappa_0 k_1) = c_0 = m_T(\kappa_0 k_2).
    \end{align*}
We further have $m_T'(\kappa_0 k_1)<0<m_T'(\kappa_0 k_2)$, and the asymptotic behaviour of $c_0=c_0(T ;k_1,k_2)$ and $\kappa_0=\kappa_0(T;k_1,k_2)$ as $T\downarrow 0$ is given by
    \begin{align*}
        c_0 \sim T^{1/4}\sqrt{\sqrt{\frac{k_1}{k_2}}+\sqrt{\frac{k_2}{k_1}}} \quad \text{ and }\quad \kappa_0\sim \frac{1}{\sqrt{k_1k_2 T}},
    \end{align*}
while their asymptotic behaviour as $T\uparrow \frac{1}{3}$ is given by
\begin{align*}
       c_0\sim 1 \quad\text{ and }\quad \kappa_0 \sim \sqrt{\left(\frac{1}{3}-T\right)\frac{45}{k_1^2+k_2^2}}.
\end{align*}
\end{proposition}
\begin{proof}
    We begin by solving
    \begin{equation}\label{eq:bifpoint}
        m_T(k_1\kappa)=m_T(k_2\kappa)
    \end{equation}
    in $\kappa>0$. By Lemma \ref{lemma:mTproperties} there is a number $\xi_T\in (0,\infty)$ such that $m'_T$ is negative on $(0,\xi_T)$ and positive on $(\xi_T,\infty)$. Since $k_1<k_2$ we find that $m_T(k_1\kappa)>m_T(k_2\kappa)$ for $\kappa\leq \xi_T/k_2$ and $m_T(k_1\kappa)<m_T(k_2\kappa)$ for $\xi_T/k_1\leq \kappa$, so there exists a $\kappa_0\in (\frac{\xi_T}{k_2},\frac{\xi_T}{k_1})$ solving \cref{eq:bifpoint}. Moreover, $m_T'(k_1\kappa)<0$ and $m_T'(k_2\kappa)>0$ for $\kappa\in (\frac{\xi_T}{k_2},\frac{\xi_T}{k_1})$, which means that $\kappa_0$ is unique and, in particular, $m_T'(k_1\kappa_0)<0$ and $m_T'(k_2\kappa_0)>0$. 
    
    Since $\kappa_0\in (\frac{\xi_T}{k_2},\frac{\xi_T}{k_1})$ we immediately get from the behaviour of $\xi_T$ (Lemma \ref{lemma:mTproperties}) that
    \begin{equation}\label{eq:kappa0limits}
        \lim_{T\downarrow 0} \kappa_0=\infty,\qquad \lim_{T\uparrow \frac{1}{3}} \kappa_0=0,
    \end{equation}
    however, for future reference we want more precise asymptotic behaviour of $\kappa$.
    To this end we note that \cref{eq:bifpoint} is equivalent to
    \begin{equation}\label{eq: definitionOfVarphi}
         T\kappa^2k_1k_2=\frac{k_1\tanh{\kappa k_2}-k_2\tanh{\kappa k_1}}{k_1\tanh{\kappa k_1}-k_2\tanh{\kappa k_2}}\eqqcolon \varphi(\kappa).
    \end{equation}
    Using \cref{eq:kappa0limits}, we read from \cref{eq: definitionOfVarphi} that $\lim_{T\downarrow 0}T\kappa_0^2k_1k_2=\lim_{T\downarrow 0} \varphi(\kappa_0)=1$, or equivalently
    \[
        \lim_{T\downarrow 0}\kappa_0\sqrt{Tk_1k_2}=1.
    \]
    By Taylor expansion, we have  
    \begin{equation}\label{eq:phiexp1/3}
        \varphi(\kappa)=\frac{1}{3}\kappa^2k_1k_2-\frac{1}{45}\kappa^4k_1k_2(k_1^2+k_2^2)+\mathcal{O}(\kappa^6)\qquad\text{as }\kappa\to0.
    \end{equation}
    Replacing the left hand side of \cref{eq:phiexp1/3} by the left-most expression in \cref{eq: definitionOfVarphi} we get, after dividing each side by $\kappa^2k_1k_2$ and some rewriting,
    \[
        \kappa^2_0=\left(\frac{1}{3}-T\right)\frac{45}{k_1^2+k_2^2}+\mathcal{O}(\kappa^4_0)\qquad \implies\qquad  \lim_{T\uparrow \frac{1}{3}} \frac{\kappa_0}{\sqrt{\left(\frac{1}{3}-T\right)\frac{45}{k_1^2+k_2^2}}}=1,
    \]
    where the implication follows from the already established $\lim_{T\uparrow \frac{1}{3}}\kappa_0=0$ by \cref{eq:kappa0limits}.
    The asymptotic behaviour of $c_0$ then follows from that of $\kappa_0$ and the identity $c_0=m_T(k_1\kappa_0)$ or equivalently $c_0=m_T(k_2\kappa_0)$.
\end{proof}

\subsection{Asymmetry of sums of sinusoids}\label{sec: asymmetryOfSinusoids}
This subsection is dedicated to characterizing the symmetric and asymmetric elements in the space
\begin{align}\label{eq: theSpaceVForAsymmetryDiscussion}
    V = \spann \{\cos(k_1 x),\cos(k_2 x),\sin(k_1 x),\sin(k_2 x)\}.
\end{align}
This turns out to be easier when elements $v\in V$ are expressed in a kind of polar form; see \cref{eq: theRepresentationOfElementsInV} below. Note that this representation covers all of $V$ since applying the cosine sum formula to \cref{eq: theRepresentationOfElementsInV} gives
\begin{align*}
v(x)=&\,\overbrace{r_1\cos(k_1\theta_1)}^{\tau_1}\cos(k_1x)+\overbrace{r_2\cos(k_2\theta_2)}^{\tau_2}\cos(k_2x)\\
   &\, -\underbrace{r_1\sin(k_1\theta_1)}_{\tau_3}\sin(k_1x)-\underbrace{r_2\sin(k_2\theta_2)}_{\tau_4}\sin(k_2x),
\end{align*}
and by varying $r_1,r_2,\theta_1,\theta_2$ we can obtain any vector $\tau\in\R^4$.

\begin{proposition}[Asymmetry condition]\label{proposition: criterionForAsymmetryOfV}
    Let the integers $1\leq k_1<k_2$ be coprime and let $r_1,r_2,\theta_1,\theta_2\in\R$. Then the function
    \begin{equation}\label{eq: theRepresentationOfElementsInV}
v(x)=r_1\cos\Big(k_1(x+\theta_1)\Big)+r_2\cos\Big(k_2(x+\theta_2)\Big)
    \end{equation} is asymmetric if, and only if, $r_1r_2\neq 0$ and $\theta_1-\theta_2\notin \frac{\pi}{k_1k_2}\Z$.
\end{proposition}
\begin{proof}
    That both $r_1,r_2$ have to be nonzero follows as a single translated cosine is symmetric (according to Definition \ref{def:asymmetric}). Thus, we assume from here on that $r_1r_2\neq 0$ and turn to the criterion $\theta_1-\theta_2\neq \frac{\pi}{k_1k_2}\Z$. For necessity, assume $x\mapsto v(x+a)$ is symmetric for some $a\in\R$. This forces $\sin(k_1 a)=\sin(k_2 a)=0$ (since we must have $v'(a)=v'''(a)=0$) from which we conclude that there are integers $m_1,m_2\in\Z$ such that $k_1(a + \theta_1)=m_1\pi $ and $ k_2(a + \theta_2)=m_2\pi$. In particular, we get the desired conclusion because
    \begin{equation*}
          \theta_1-\theta_2=\frac{(m_1k_2 -m_2k_1)\pi}{k_1k_2}.
    \end{equation*}
    Next, for sufficiency, suppose $\theta_1-\theta_2=\frac{n\pi }{k_1k_2}$ for some $n\in\Z$. As $k_1$ and $k_2$ are coprime, we can find $m_1,m_2\in\Z$ such that $n=m_1k_2-m_2k_1$. Then, setting
    \begin{align*}
        a\coloneqq  \frac{m_1\pi}{k_1}-\theta_1= \frac{m_2\pi }{k_2}-\theta_2,
    \end{align*}
we see that $x\mapsto v(x+a)=r_1\cos(m_2\pi)\cos(k_1x) +r_2\cos(m_1\pi)\cos(k_2x)$ is symmetric.
\end{proof}

\begin{remark}\label{rem: describingTheTopologyOfV}
  We here give a brief visual description of the previous proposition. For a point $r=(r_1,r_2)$ in the quarter plane $[0,\infty)^2$ and a point $\theta=(\theta_1,\theta_2)$ on the two-dimensional torus $\Theta\coloneqq \big(\R/\tfrac{2\pi}{k_1}\Z\big)\times \big(\R/\tfrac{2\pi}{k_2}\Z\big)$, let $v(r,\theta)$ denote the function described by \cref{eq: theRepresentationOfElementsInV}; any element of \cref{eq: theSpaceVForAsymmetryDiscussion} admits such a representation. For simplicity, assume for now that $(k_1,k_2)=(2,5)$; \Cref{fig:Vshape} depicts the domains $[0,\infty)^2$ and $\Theta$ with the regions $r_1r_2=0$ and $\theta_1-\theta_2\in \frac{\pi}{10}\Z$ marked blue.

  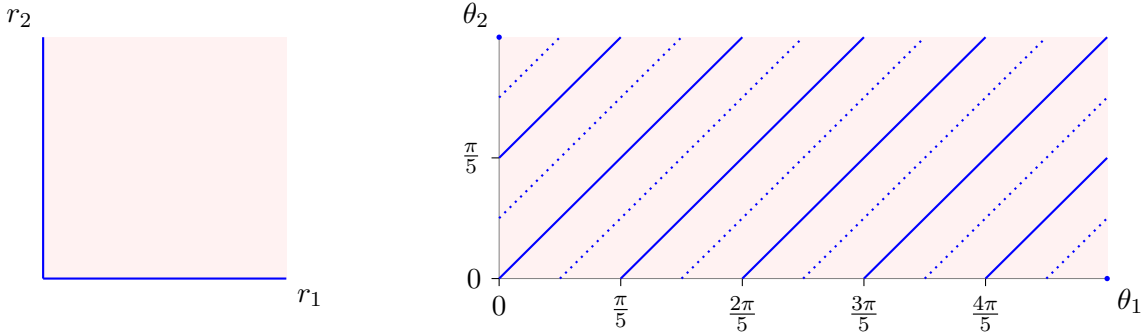
\begin{figure}[h!]
  \begin{tikzpicture}
  \fill[red!5] (0,0) rectangle (3.2,3.2);

  \draw[thick, blue] (0,0) -- (3.2,0) node[below right, black] {$r_1$};
  \draw[thick, blue] (0,0) -- (0,3.2) node[above left, black] {$r_2$};

  \begin{scope}[shift={(6,0)}]
    \fill[red!5] (0,0) rectangle (8,3.2);


 \draw[very thin, gray] (0,0) -- (8,0) node[below right, black] {$\theta_1$};
 
 \draw (0,0) -- (0,-0.1) node[below, black] {$0$};
  \draw (1.6,0) -- (1.6,-0.1) node[below, black] {$\frac{\pi}{5}$};
  \foreach \x in {2,3,4} {
    \draw ({8*\x/5},0) -- ({8*\x/5},-0.1) node[below, black] {$\frac{\x\pi}{5}$};
  }

  \draw[very thin, gray] (0,0) -- (0,3.2) node[above left, black] {$\theta_2$};
  \draw (0,0) -- (-0.1,0) node[left, black] {$0$};
  \draw (0,1.6) -- (-0.1,1.6) node[left, black] {$\frac{\pi}{5}$};

        \draw[thick, blue] (0,1.6) -- (1.6,3.2);
     \draw[thick, blue] (6.4,0) -- (8,1.6);

     \fill[blue] (0,3.2) circle (1pt); 
      \fill[blue] (8,0) circle (1pt);
   \foreach \i in {0,1,...,3} {
    \draw[thick, blue] ({1.6*\i},0) -- ({1.6*\i + 3.2},3.2);
  }
  \draw[dotted, thick, blue] (0,2.4) -- (0.8,3.2);
 \draw[dotted, thick, blue] (0,0.8) -- (2.4,3.2);
   \draw[dotted, thick, blue] (5.6,0) -- (8,2.4);
 \draw[dotted, thick, blue] (7.2,0) -- (8,0.8);
   \foreach \i in {0,1,...,2} {
    \draw[dotted, thick, blue] ({1.6*\i + 0.8},0) -- ({1.6*\i + 4},3.2);
  }
  \end{scope}
\end{tikzpicture}
\caption{Respective plots of the domains $[0,\infty)^2$ and $\Theta=\left(\R/{\pi}\Z\right)\times \left(\R/\tfrac{2\pi}{5}\Z\right)$. Blue regions denote points where $r_1r_2=0$ or $\theta_1-\theta_2\in \frac{\pi}{10}\Z$; the latter set can be split in $\theta_1-\theta_2\in \frac{\pi}{10}(2\Z)$ and $\theta_1-\theta_2\in \frac{\pi}{10}(\Z\setminus2\Z)$ marked by solid and dotted lines respectively.}
\label{fig:Vshape}
\end{figure}

\noindent By the previous proposition, $v(r,\theta)$ is asymmetrical if, and only if, both $r$ and $\theta$ have been chosen from a red region. The blue region in $\Theta$ consists of two disjoint circles (in the topology of $\Theta$): One is solid and corresponds to $\theta_1-\theta_2\in \frac{\pi}{10}(2\Z)$ and the other is dotted and corresponds to $\theta_1-\theta_2\in \frac{\pi}{10}(\Z\setminus 2\Z)$. Note that an $x$-translation $x\mapsto x + a$ of the function $v$ corresponds to a diagonal shift $(\theta_1,\theta_2)\mapsto (\theta_1+a,\theta_2+a)$ on $\Theta$, and so if $\theta\in \Theta$ lies on the solid/dotted circle, so will $\theta + a$. In particular, one obtains all elements in $V$ `up to translation' by restricting $\theta\in\Theta$ to some vertical (or horizontal) line-segment connecting the solid circle with itself.

The red region in $\Theta$ consist of two disjoint bands separated by the two blue circles. Each band is homeomorphic to the Cartesian product $(0,1)\times \T$. As for $[0,\infty)^2$, its red region is $(0,\infty)^2$. And so, the asymmetrical elements in $V$ from \cref{eq: theSpaceVForAsymmetryDiscussion} consists of two disjoint open sets, each homeomorphic to $(0,\infty)^2\times (0,1) \times \T $ (which is homeomorphic to $\R^3\times \T$). 

While we restricted our attention to $(k_1,k_2)=(2,5)$ for visual simplicity, the same conclusion holds for all coprime pairs: The asymmetrical elements of $V$ constitutes two regions, each homeomorphic to $\R^3\times \T$. We omit the proof as this fact is irrelevant to our analysis.\qedhere

\end{remark}

\subsection{Symmetry-preservation of Lyapunov--Schmidt reduction}
We here present various symmetry-properties regarding the Lyapunov--Schmidt reduction sketched out in Section \ref{sec: SetupAndMainResults}. Several details from said section will matter little for what follows. So for generality sake, we here assume only that $V$ is a closed subspace of $L^2(\T)$ which is closed under $x$-translations and spanned by a finite number of sinusoids. We denote by $W$ its orthogonal complement and set $W^s\coloneqq \mathcal{C}^s(\T)\cap W$ for $s\geq 0$. Finally, $P_V$ and $P_W$ will denote the orthogonal projections from $L^2(\T)$ to $V$ and $W$ respectively.

\begin{proposition}[Preservation of symmetry]\label{prop: UIsSymmetricIfAndOnlyIfVIs}
    Let $c,\kappa,T>0$ and let $u\in \mathcal{C}^s(\R)$, where $s>1$. Decompose $u$ into $v\coloneqq P_V u$ and $w\coloneqq P_W u$, and suppose that $w$ is the unique solution of \cref{eq: equationForW} in some zero-centered ball $B_r\subset W^s$. Then $u$ is symmetric if, and only if, $v$ is.
\end{proposition}
\begin{proof}
For sufficiency, suppose that $v$ is symmetric. As $V$ is closed under translations, so is $W$, and we conclude that the operators $P_V$ and $P_W$ commute with $x$-translations. The same is true for the Fourier multiplier $M_{T,\kappa}-c$ and so the equation \cref{eq: equationForW} is translation symmetric. At the cost of translating $u,v,w$ we may then assume, without loss of generality, that $v$ is symmetric about zero. Consider now the reflection operator $\mathcal{R}\colon f(x)\mapsto f(-x)$. On the Fourier side, the operators $P_V$, $P_W$ and $M_{T,\kappa}-c$ correspond to real valued multipliers while $\mathcal{R}$ corresponds to the act of complex conjugation; thus, the operators commute. Applying then $\mathcal{R}$ to \cref{eq: equationForW}, we find that $\mathcal{R}w$ also satisfies said equation (since $\mathcal{R}v=v$). And because $\|\mathcal{R}w\|_{\mathcal{C}^s(\T)}=\|w\|_{\mathcal{C}^s(\T)}$ we conclude that $\mathcal{R}w=w$ due to the uniqueness assumption. Thus
        \begin{equation*}
\mathcal{R}u = \mathcal{R}v+\mathcal{R}w=v+w=u.
    \end{equation*}
For necessity, suppose, without loss of generality, that $u$ is symmetric about zero. Using that $\mathcal{R}$ commutes with $P_V$ we then get
    \begin{equation*}
\mathcal{R}v=\mathcal{R}P_Vu=P_V\mathcal{R}u=P_Vu=v.\qedhere
    \end{equation*}
\end{proof}
\begin{corollary}\label{cor: simpleBifurcationPointsGiveRiseToOnlySymmetricSolutions}
    Simple bifurcation points of \cref{eq: scaledSteadyCapillaryWhithamEquation} locally only give rise to symmetric solutions. Thus, in the strong-surface tension regime, $T\geq \frac{1}{3}$, the solutions of \cref{eq: scaledSteadyCapillaryWhithamEquation} are all symmetric near bifurcation points.
\end{corollary}
\begin{proof}
    For simple bifurcation points $(c_0,\kappa_0,T_0)$ the kernel of $D_u J(0;c_0,\kappa_0,T_0)$ is given by 
    \begin{equation*}
        V= \spann\{\cos(x),\sin(x)\},
    \end{equation*}
    where we have chosen the appropriate scaling $\kappa_0>0$ as described by Remark \ref{rem: oneMayAssumeThatK1AndK2AreRelativePrimes}. Any element $v\in V$ admits a representation of the form $v(x)=r\cos(x+\theta)$ for some $r,\theta\in \R$; thus, all elements of $V$ are symmetric. The proof of Proposition \ref{prop: infdimAKAuniqueWSolution} works for simple bifurcation points as well, and so by Proposition \ref{prop: UIsSymmetricIfAndOnlyIfVIs} we conclude the symmetry of solutions found sufficiently close to the bifurcation point $(0,c_0,\kappa_0,T_0)$. 
    
    The last part of the corollary then follows from Corollary \ref{cor: allBifurcationPointsAreSimpleInTheStrongSurfaceTensionRegime} which tells us that every bifurcation point is simple in the strong-surface tension regime.
\end{proof}
\begin{remark}
    The above argument can equally be applied to the Whitham equation (where $T=0$); its periodic travelling wave solutions are thus necessarily symmetric near bifurcation points.
\end{remark}
\begin{remark}\label{remark: specialBifurcationForCEqualToOne}
    There is a special bifurcation taking place for the value $c_0=1$ and any $T_0>0$; the kernel of $D_uJ(0)$ can then, after selecting an appropriate scaling factor $\kappa_0>0$, be written
\begin{equation*}
    V=\spann\{1, \cos(x),\sin(x)\}.
\end{equation*}
    Such a bifurcation point is not simple, yet all the elements of $V$ are symmetric, and so the resulting solutions are also symmetric by \Cref{prop: UIsSymmetricIfAndOnlyIfVIs}.
\end{remark}

Next, with $J$ as in \cref{eq: scaledSteadyCapillaryWhithamEquation}, its translation symmetry and variational structure implies the following orthogonality property which enhances our Lyapunov--Schmidt reduction argument (see Section \ref{sec: lyapunovSchmidtArgument}). This interplay between symmetry and Lyapunov--Schmidt reduction is part of the more general framework found in \cite{Golubitsky1985}.
\begin{proposition}
   Let $u\in \mathcal{C}^s(\T)$ with $s>1$, and let $c,\kappa,T\in \R_+$. Then $J(u;c,\kappa,T)$ is orthogonal to $u' = \partial_x u$. That is
    \begin{equation*}
        \langle J(u),u'\rangle\coloneqq \frac{1}{2\pi}\int_{0}^{2\pi}J(u)u'\,dx=0.
    \end{equation*}
\end{proposition}
\begin{proof}
    Suppose first $u\in \mathcal{C}^s(\T)$ with $s>\frac{3}{2}$. We can leverage the variational structure of $J$ by introducing the functional
\begin{equation*}
		\mathcal{J}(u;c,\kappa,T)\coloneqq \frac{1}{2\pi}\int_0^{2\pi} \left( \frac{u( M_{T,\kappa}-c) u}{2}+\frac{u^3}{3}\right) dx   . 
\end{equation*}
As $M_{T,\kappa}-c$ is a symmetric operator, the Fréchet derivative of $\J$ with respect to $u$ in $C^{s-1}(\T)$ is given by
\begin{equation*}
	D_u \mathcal{J}(u)v=\frac{1}{2\pi}\int_0^{2\pi} \Big((M_{T,\kappa}-c)u+u^2\Big)v\, dx=
	\langle J(u),v\rangle
\end{equation*}
for any $v\in \mathcal{C}^{s-1}(\T)$. And since the limit $\lim_{\varepsilon\to 0}\frac{u(\cdot+\varepsilon)-u}{\varepsilon}= u'$ holds in $\mathcal{C}^{s-1}(\T)$, we get
\begin{equation*}
\langle J(u),u'\rangle=\frac{d}{da}\mathcal{J}(u(\cdot + a))\Big|_{a=0} =0,
\end{equation*}
where the latter equality holds as $\mathcal{J}$ is translation invariant. The proposition extends to $s\in(1,\tfrac{3}{2}]$ by a density argument.
\end{proof}

\begin{corollary}\label{corollary: JOfUIsOrthogonalToVPrime}
    Let $c,\kappa,T>0$ and let $u\in \mathcal{C}^s(\R)$, where $s>1$. Decompose $u$ into $v\coloneqq P_V u$ and $w\coloneqq P_W u$, and suppose that \cref{eq: equationForW} is satisfied. Then $J(u;c,\kappa,T)$ is orthogonal to $v'=\partial_x v$.
\end{corollary}
\begin{proof}
Note first that since $w'\in W^{s-1}\subset W$ (which follows as $W$ is closed under translations), and since we assume that \cref{eq: equationForW} is satisfied, we get
    \begin{equation*}
        \langle J(u),w'\rangle =  \langle P_WJ(u),w'\rangle=0.
    \end{equation*}
By the previous proposition, we then infer that
    \begin{equation*}
        0 = \langle J(u),u'\rangle = \langle J(u),v'\rangle + \langle J(u),w'\rangle=\langle J(u),v'\rangle. \qedhere
    \end{equation*}
\end{proof}

\section{Proof of the Existence Theorem}\label{sec: LocalBifurcationTheoryToSymmetryBreakingPoints}
This section is dedicated to the rigorous proof of the existence theorem under the symmetry breaking assumption from \Cref{def: symmetryBreakingBifurcationPoints}. The question of symmetry breaking is then further investigated in \Cref{sec: symmetryBreakingWavePairs}. The Lyapunov-Schmidt reduction follows in a straightforward way from the preliminary results above and most of this section concerns the finite dimensional problem. This analysis requires some technical results as the solution of the infinite dimensional problem has to be expanded to a relatively high order.

Throughout the section we consider a fixed wave number pair $1\leq k_1 < k_2$ and let $(c_0,\kappa_0,T)$ denote a corresponding double bifurcation point of \cref{eq: scaledSteadyCapillaryWhithamEquation} (Definition \ref{def: simpleAndDoubleBifrucationPoints}) where $T\in (0,\frac{1}{3})$ is considered a free parameter; by Proposition \ref{prop: c0AndKappa0AreDeterminedByK1K2AndT} both $c_0$ and $\kappa_0$ are uniquely determined by $T,k_1,k_2$. In light of Remark \ref{rem: oneMayAssumeThatK1AndK2AreRelativePrimes} we shall further assume without loss of generality that $(k_1,k_2)$ are coprime. As in Section \ref{sec: SetupAndMainResults}, we will be working with the spaces $V,W,W^s,U^s$ defined in \cref{eq: theSpaces} and $P_V$ and $P_W$ will denote the $L^2$-orthogonal projections onto $V$ and $W$ respectively.

We shall parameterize elements $v\in V$ using 
\begin{equation*}
    r=(r_1,r_2)\in [0,\infty)^2 \quad\text{ and }\quad \theta=(\theta_1,\theta_2)\in\Theta\coloneqq \big(\R/\tfrac{2\pi}{k_1}\Z\big)\times \big(\R/\tfrac{2\pi}{k_2}\Z\big),
\end{equation*}
by writing  
\begin{equation}\label{eq: parameterizationOfV}
    v(x;r,\theta)=r_1\cos\Big(k_1(x+\theta_1)\Big) + r_2\cos\Big(k_2(x+\theta_2)\Big).
\end{equation}
Such a parameterization covers all of $V$ (see Section \ref{sec: asymmetryOfSinusoids}) and makes easy to determine whether $v$ is asymmetric or not (Proposition \ref{proposition: criterionForAsymmetryOfV}). Moreover, this parameterization allows for a useful Taylor--Fourier expansion of solutions (Lemma \ref{lemma: Taylor-FourierExpansion}) from  which the symmetry breaking criterion of Definition \ref{def: symmetryBreakingBifurcationPoints} becomes evident. Finally, for any $\theta\in \Theta$ we introduce a corresponding basis of $V=\spann\{v_{\cos}^1,v_{\cos}^2,v_{\sin}^1,v_{\sin}^2\}$ given by
\begin{align*}
    v_{\cos}^1(x;\theta)\coloneqq&\,\cos(k_1(x+\theta_1)),& v_{\cos}^2(x;\theta)\coloneqq&\, \cos(k_2(x+\theta_2))\\
    v_{\sin}^1(x;\theta)\coloneqq&\,\sin(k_1(x+\theta_1)), & v_{\sin}^2(x;\theta)\coloneqq&\,\sin(k_2(x+\theta_2)).
\end{align*}
In particular, in the above parameterization, we find that 
\begin{equation}\label{eq: twoImmediateIdentitiesByTheParameterizationOfV}
    v=r_1v_{\cos}^1 + r_2 v_{\cos}^2\quad \text{ and }\quad -v'=k_1r_1v_{\sin}^1 + k_2r_2 v_{\sin}^2.
\end{equation}
We also decompose the solution space $U^s = P_VU^s\oplus P_WU^s = V\oplus W^s$ so that any $u\in U^s$ is defined by $(v,w)\in V\times W^s$ through
\begin{align*}
    P_V u\coloneqq v(r,\theta)=r_1v_{\cos}^1+r_2v_{\cos}^2,\qquad
    P_Wu\coloneqq w.
\end{align*}
We are now ready to proceed with the Lyapunov-Schmidt reduction.
\subsection{Lyapunov--Schmidt reduction}
For convenience, we here restate the system \cref{eq: rewritingTheMainEquationIntoTwoEquations}, which is the $(P_V,P_W)$-decomposition of \cref{eq: scaledSteadyCapillaryWhithamEquation}, that reads
\begin{subequations}\label{eq: restating TheMainEquationIntoTwoEquations}
    \begin{empheq}[left={\empheqlbrace\,}]{align}
      &  P_V J(v+w,c,\kappa,T)=(M_{T,\kappa}-c)v + P_V(v+w)^2=0,
        \label{eq:lsfindim} \\
      &  P_W J(v+w,c,\kappa,T)=(M_{T,\kappa}-c)w + P_W(v+w)^2=0,
        \label{eq:lsinfdim}
    \end{empheq}
\end{subequations}
We now perform the Lyapunov--Schmidt reduction argument yielding a small solution $w\in W^s$ of the infinite dimensional problem \cref{eq:lsinfdim} for (any) given $(v,c,\kappa)\approx(0,c_0,\kappa_0)$ and $T\in(0,\frac{1}{3})$.
\begin{proposition}[Solving for $w$]\label{prop: infdimAKAuniqueWSolution} Let the integers $1\leq k_1<k_2$ be coprime. There is a zero-centered ball $B\subset W^s$ such that: For every $T\in(0,1/3)$, $\vert v\vert \ll 1$, and $(c,\kappa)$ satisfying $\vert (c-c_0,\kappa-\kappa_0)\vert\ll 1$, there exists a unique solution $w=w(v,c,\kappa,T)\in B$ to \cref{eq:lsinfdim} depending analytically on its arguments.
\end{proposition}
\begin{proof}
We can restrict $P_WJ$ to $W^s$ and consider it as an operator $P_WJ:W^s\times \mathbb{R}^2\times (0,1/3)\to W^{s-1/2}$. Clearly $P_WJ(0;c_0,\kappa_0,T)=0$ and since
\[
D_{w}(P_WJ)(0;c_0,\kappa_0,T)=P_W(D_uJ)(0;c_0,\kappa_0,T)\vert_{W_s},
\]
which is an isomorphism by Proposition \ref{prop: theLinearizationIsFredholm}, we can apply the analytic implicit function theorem \cite[Theorem 4.5.4]{Buffoni_2003} to obtain $w(v,c,\kappa,T)$ for all $T\in (0,1/3)$, $\vert r\vert\ll 1$ and $(c,\kappa)$ sufficiently close to $(c_0,\kappa_0)$.
\end{proof}

\subsection{The finite dimensional problem}It remains to solve \cref{eq:lsfindim} with $w=w(v,c,\kappa,T)$ being the solutions of \cref{eq:lsinfdim} from \Cref{prop: infdimAKAuniqueWSolution}.
Defining the \emph{reduced operator}
\begin{equation}\label{eq: theReducedEquation}
    I(r,\theta,c,\kappa,T)\coloneqq J\Big(v(r,\theta)+w\big(v(r,\theta),c,\kappa,T\big);c,\kappa,T\Big)
\end{equation}
we can write \cref{eq:lsfindim} (with $w$ from Proposition \ref{prop: infdimAKAuniqueWSolution}) as the four scalar equations
   \begin{subequations}\label{eq: finalSystem}
    \begin{empheq}[left={\empheqlbrace\,}]{align}
      &  \big\langle I(r,\theta,c,\kappa,T),v_{\cos}^1\big\rangle=0, \label{eq:findimEq1}\\
      & \big\langle I(r,\theta,c,\kappa,T),v_{\cos}^2\big\rangle=0, \label{eq:findimEq2}\\
       & \big\langle I(r,\theta,c,\kappa,T),v_{\sin}^1\big\rangle=0,\label{eq:findimEq3}\\
       & \big\langle I(r,\theta,c,\kappa,T),v_{\sin}^2\big\rangle=0,\label{eq:findimEq4}
    \end{empheq}
\end{subequations}
    since $v_{\cos}^1$, $v_{\sin}^1$, $v_{\cos}^2$, and $v_{\sin}^2$ form a basis for $V$.
However, \cref{eq:findimEq3} and \cref{eq:findimEq4} are equivalent: By \Cref{corollary: JOfUIsOrthogonalToVPrime} and \cref{eq: twoImmediateIdentitiesByTheParameterizationOfV}
we have the relation
\begin{equation}\label{eq: theLinearDependencyOfTheTwoSinEquationsWrittenOutForTheFiniteDimensionalProblem}
    r_1k_2\langle I,v_{\sin}^1\rangle+r_2k_2\langle I,v_{\sin}^2\rangle=0.
\end{equation}
    Thus, at least when $r_1r_2\neq0$, we conclude that $\langle I,v_{\sin}^1\rangle=0\Longleftrightarrow\langle I,v_{\sin}^2\rangle=0$; in fact, this equivalence holds also when $r_1r_2=0$ as can be seen by \Cref{proposition:lastfineq}.
    
   Seeking to solve \cref{eq: finalSystem}, we shall first construct $c=c(r,\theta,T)$ and $\kappa=\kappa(r,\theta,T)$ such that \cref{eq:findimEq1,eq:findimEq2} are solved for any $\vert r\vert\ll 1$, $\theta \in \Theta$ and $T\in(0,1/3)$. Then, substituting for $c$ and $\kappa$ in \cref{eq:findimEq3} leaves us with one remaining equation in the variables $(r,\theta,T)$. This final equation is trivially satisfied for symmetric $v$ (a consequence of Proposition \ref{prop: UIsSymmetricIfAndOnlyIfVIs}) but often impossible to solve for asymmetric $v$. To accurately identify when it \textit{can} be solved, we will use an appropriate power-expansion of the equation. As this expansion will also be useful for solving \cref{eq:findimEq1,eq:findimEq2}, we proceed by introducing said expansion which is naturally indexed over multi-indices:
   
   To this end we use the notation $\alpha,\beta \in \N_0^2$ for multi-indices. 
    We use the multi-index convention that $|\alpha|=|\alpha_1|+|\alpha_2|$ and $r^{\alpha}=r_1^{\alpha_1}r_2^{\alpha_2}$. We also introduce 
    \begin{equation}\label{eq: DefinitionOfTheMultiindexExponential}
E\coloneqq(e^{ik_1(x+\theta_1)},e^{ik_2(x+\theta_2)}),
 \end{equation}
 so that, in particular,
    \begin{equation*}
         E^{\alpha\pm\beta} = e^{ik_1(x+\theta_1)(\alpha_1\pm\beta_1) + ik_2(x+\theta_2)(\alpha_2\pm\beta_2)}.
    \end{equation*}
    With this notation, we shall expand a solution $u=v+w$ in a Taylor-Fourier series; the series obtained by first expanding $u$ as a Taylor series in $r$ and then expanding every Taylor-coefficient in its Fourier series. Conveniently, using that $u$ solves \cref{eq: scaledSteadyCapillaryWhithamEquation} we get a recursion relation for the coefficients in this expansion. This recursion is easily expressed with the help of the operator
\begin{equation}\label{eq: theDefinitionOfL}
     L\coloneqq (c-M_{T,\kappa})^{-1}P_W
\end{equation}
    which is a Fourier multiplier with symbol
\begin{equation}\label{eq: theDefinitionOfTheMultiplierEll}
    \ell(k)=\left\{
                \begin{aligned}
                    &\left(c-m_{T}(\kappa k)\right)^{-1}\qquad &&k\notin \{\pm k_1,\pm k_2\},\\
                    &0&& k\in \{\pm k_1,\pm k_2\}.
                \end{aligned}
             \right.
\end{equation}
The implicit dependence of $L$ and $\ell$ on $c,\kappa,T>0$ has here been suppressed, but note that the objects are well defined well-defined whenever $(c,\kappa)\approx (c_0,\kappa_0)$. 
    \begin{lemma}[Taylor-Fourier expansion of the solution]\label{lemma: Taylor-FourierExpansion}
        Let $(u,c,\kappa,T)\in U^s\times\R^{3}_+$ be a solution of \cref{eq: scaledSteadyCapillaryWhithamEquation} near the bifurcation point $(0,c_0,\kappa_0,T)$.
        Then there are coefficients $\hat{u}_{\alpha,\beta}\in \R$ indexed over multi-indices $\alpha,\beta\in \N_0^2$ and dependent on only $c,\kappa,T$ such that 
        \[
            u=v+w=\sum_{\vert\alpha\vert+\vert\beta\vert\geq 1} \hat{u}_{\alpha,\beta}r^{\alpha+\beta}E^{\alpha-\beta},
        \]
       where $E$ is as in \cref{eq: DefinitionOfTheMultiindexExponential}. The corresponding coefficients for $u^2$ will be denoted $\widehat{u^2}_{\alpha,\beta}$ so that
\begin{equation}\label{eq:u2expansion}
            u^2=(v+w)^2=\sum_{\vert\alpha\vert+\vert\beta\vert\geq 2} \widehat{u^2}_{\alpha,\beta}r^{\alpha+\beta}E^{\alpha-\beta}.
        \end{equation}
        The coefficients are symmetric in their indices, so that $\hat{u}_{\alpha,\beta}=\hat{u}_{\beta,\alpha}$ and $\widehat{u^2}_{\alpha,\beta}=\widehat{u^2}_{\beta,\alpha}$.
    \end{lemma}
    \begin{proof}
    The first four coefficients, which corresponds to $v=v(r,\theta)$, are given by 
        \[
            \hat{u}_{(1,0),(0,0)}=\hat{u}_{(0,0),(1,0)}=\hat{u}_{(0,1),(0,0)}=\hat{u}_{(0,0),(0,1)}=\frac{1}{2},
        \]
    as follows from \cref{eq: parameterizationOfV} and the identity $\cos(kx)=\tfrac{1}{2}(e^{ikx}+e^{-ikx})$. Using then that
        \begin{equation}\label{eq:urecursive}
            w=L(v+w)^2,
        \end{equation}
       where $L$ is given by \cref{eq: theDefinitionOfL}, we can compute the rest of the coefficients recursively:
        \[
            \hat{u}_{\alpha,\beta}E^{\alpha-\beta}=L\sum_{\substack{\alpha'+\alpha''=\alpha\\\beta'+\beta''=\beta}} \hat{u}_{\alpha',\beta'}\hat{u}_{\alpha'',\beta''}E^{\alpha-\beta},
        \]
        or equivalently
        \begin{equation}\label{eq:urecursive2}
            \hat{u}_{\alpha,\beta}=\sum_{\substack{\alpha'+\alpha''=\alpha\\\beta'+\beta''=\beta}} \hat{u}_{\alpha',\beta'}\hat{u}_{\alpha'',\beta''}\ell\big(k_1(\alpha_1-\beta_1)+k_2(\alpha_2-\beta_2)\big),
        \end{equation}
        where the multiplier $\ell$ is defined in \cref{eq: theDefinitionOfTheMultiplierEll}.
    The coefficients $\widehat{u^2}_{\alpha,\beta}$ can then be computed from the formula
        \begin{equation}\label{eq:u2recursive}
\widehat{u^2}_{\alpha,\beta}=\sum_{\substack{\alpha'+\alpha''=\alpha\\\beta'+\beta''=\beta}} \hat{u}_{\alpha',\beta'}\hat{u}_{\alpha'',\beta''}.
        \end{equation}

        The symmetry of the coefficients follow from the symmetry of $\hat{u}_{(1,0),(0,0)}$, $\hat{u}_{(0,0),(1,0)}$, $\hat{u}_{(0,1),(0,0)}$ and $\hat{u}_{(0,0),(0,1)}$ together with the recursion relations and the fact that $k\mapsto \ell(k)$ is an even function so that $\ell(k_1(\alpha_1-\beta_1)+k_2(\alpha_2-\beta_2))=\ell(k_1(\beta_1-\alpha_1)+k_2(\beta_2-\alpha_2))$.
    \end{proof}
    With this expansion we can show it is possible to bring out a factor $r_1$ and $r_2$ from $\langle I,v_{\cos}^1\rangle$ and $\langle I,v_{\cos}^2\rangle$ respectively, which allows us to solve the equations $\langle I,v_{\cos}^1\rangle=\langle I,v_{\cos}^2\rangle=0$ with the implicit function theorem. This factorization is only possible if we avoid the case when $k_2$ is an integer multiple of $k_1$; for re-scaled coprime wave numbers this means we must assume $k_1\neq 1$. While symmetric waves can be found for $k_1=1$ (see \cite{Ehrnstroem2019}), the presence of small asymmetrical solutions of \cref{eq: scaledSteadyCapillaryWhithamEquation} actually necessitates that $k_1>1$ (see \Cref{prop:NoSymmetryBreaking}), and so continuing with this assumption is unproblematic. 
    
    \begin{proposition}[A factorization result]\label{proposition:Psifactorization}
 	Let the integers $2\leq k_1<k_2$ be coprime and the function $I(r,\theta,c,\kappa,T)$ be as in \cref{eq: theReducedEquation}. Then there are analytic functions $\Psi_1=\Psi_1(r,\theta,c,\kappa,T)$ and $\Psi_2=\Psi_2(r,\theta,c,\kappa,T)$, such that
	\begin{align*} 
		\langle I,v_{\cos}^1\rangle=r_1\Psi_1,\\
		\langle I,v_{\cos}^2\rangle=r_2\Psi_2,
	\end{align*}
    and 
    \begin{align}
    \Psi_1(r,\theta,c_0,\kappa_0,T)&=m_{T}(\kappa k_1)-c+\mathcal{O}(\vert r\vert^2),\label{eq:Psi1rexp}\\
    \Psi_2(r,\theta,c_0,\kappa_0,T)&=m_{T}(\kappa k_2)-c+\mathcal{O}(\vert r\vert^2).\label{eq:Psi2rexp}
    \end{align}
    \end{proposition}
    \begin{proof} We prove the result for $\langle I,v_{\cos}^1\rangle$ and omit the proof for $\langle I,v_{\cos}^2\rangle$ as it is completely analogous. We begin by expressing $\langle I,v_{\cos}^1\rangle$ in two parts
    \begin{align*}
    \langle I,v_{\cos}^1\rangle &=\langle J(v+w),v_{\cos}^1\rangle
    =\langle D_{u}J(0;c,\kappa,T)[v+w],v_{\cos}^1\rangle
    +\langle(v+w)^2,v_{\cos}^1\rangle.
    \end{align*}
    For the first term we obtain
    \[
    \langle D_{u}J(0;c,\kappa,T)[v+w],v_{\cos}^1\rangle=r_1 (m_{T}(\kappa k_1)-c)
    \]
    because $w\in W$ so only the $r_1 v_{\cos}^1 $ term in $v$ gives a contribution. Moreover we note that $c_0=m_{T}(\kappa_0k_1)$. 
    
    For the second term we can use the expansion of $(v+w)^2$ in \cref{eq:u2expansion} to find
    \[
        \langle(v+w)^2,v_{\cos}^1\rangle=\sum_{\vert\alpha\vert+\vert\beta\vert\geq 2}r^{\alpha+\beta}\widehat{u^2}_{\alpha,\beta}\langle E^{\alpha-\beta},v_{\cos}^1\rangle.
    \]
    There is a factor $r_1$ in every nonzero term if
    \[
       \langle E^{\alpha-\beta},v_{\cos}^1\rangle=0.
    \]
    For all $\alpha$, $\beta$ such that $\alpha_1=\beta_1=0$.
    However, $\alpha_1=\beta_1=0$ and $\langle E^{\alpha-\beta},v_{\cos}^1\rangle\neq 0$ means 
    \[
        (\alpha_2-\beta_2)k_2+k_1=0\qquad\text{or}\qquad(\alpha_2-\beta_2)k_2-k_1=0,
    \]
    but either of these equations would mean $k_2$ divides $k_1$ which contradicts the fact that they are coprime and greater or equal to $2$. Similarly, for $\vert\alpha\vert+\vert\beta\vert=2$ we see that in fact we also have $\langle E^{\alpha-\beta},v_{\cos}^1\rangle=0$ giving \cref{eq:Psi1rexp}.
    \end{proof}
By Propositions \ref{proposition: criterionForAsymmetryOfV} and \ref{prop: UIsSymmetricIfAndOnlyIfVIs}, asymmetry of the solution necessitates that $r_1r_2=0$. Thus, when seeking to find asymmetric solutions, solving $\Psi_1=\Psi_2=0$ is equivalent to solving $\langle I,v_{\cos}^1\rangle=\langle I,v_{\cos}^1\rangle=0$. We next prove a factorization result for $\langle I,v_{\sin}^1\rangle$ and $\langle I,v_{\sin}^2\rangle$ as well. Note that this factorization is also valid for $k_1=1$, which allows us to later disprove the existence of asymmetric waves for $k_1=1$.

    \begin{proposition}[A second factorization result]\label{proposition:lastfineq} Let the integers $1\leq k_1<k_2$ be coprime. there exist analytic functions $\Psi_3=\Psi_3(r,\theta,c,\kappa,T)$ and $\Psi_4=\Psi_4(r,\theta,c,\kappa,T)$, such that
		\begin{align*}
			\langle I,v_{\sin}^1\rangle&=r_1^{k_2-1}r_2^{k_1}\sin\big(k_1k_2(\theta_1-\theta_2)\big)\Psi_3,\\
            \langle I,v_{\sin}^2\rangle&=r_1^{k_2}r_2^{k_1-1}\sin\big(k_1k_2(\theta_2-\theta_1)\big)\Psi_4.
		\end{align*}
    These two functions are linearly dependent $k_1\Psi_3 + k_2\Psi_4=0$, and 
    \begin{align}
        \Psi_3(r,\theta,c,\kappa,T)&=\widehat{u^2}_{(k_2-1,0),(0,k_1)}(c,\kappa,T)+\mathcal{O}(\vert r\vert^2),\label{eq:Psi3rexp}\\
        \Psi_4(r,\theta,c,\kappa,T)&=\widehat{u^2}_{(k_2,0),(0,k_1-1)}(c,\kappa,T)+\mathcal{O}(\vert r\vert^2),\label{eq:Psi4rexp}
    \end{align}
    where the two $\widehat{u^2}$-coefficients are as described by Proposition \ref{lemma: Taylor-FourierExpansion} with their $c,\kappa,T$ dependence explicitly stated. In particular, the $\theta$-dependence of $\Psi_3$ and $\Psi_4$ vanishes when $r=0$.
	\end{proposition}
    \begin{proof}
       Because $v_{\sin}^1$ is $L^2(\mathbb{T})$-orthogonal to $v+w$ we find that
        \begin{align*}
            \langle I(r,\theta,c,\kappa,T),v_{\sin}^1\rangle=\langle (u+v)^2,v_{\sin}^1\rangle.
        \end{align*}
        Since
        \[
            (v+w)^2=\sum_{\vert\alpha\vert+\vert\beta\vert\geq 2} r^{\alpha+\beta}\widehat{u^2}_{\alpha,\beta}E^{\alpha-\beta}
        \]
        we consider
        \begin{align*}
            I_{\alpha,\beta}(\theta,c,\kappa,T)&\coloneqq\langle\widehat{u^2}_{\alpha,\beta}E^{\alpha-\beta},v_{\sin}^1\rangle.
        \end{align*}
        This can be written as
        \begin{align*}
            I_{\alpha,\beta}&=-\frac{\widehat{u^2}_{\alpha,\beta}\langle E^{\alpha-\beta}, e^{ik_1(\cdot+\theta_1)}\rangle-\widehat{u^2}_{\alpha,\beta}\langle E^{\alpha-\beta}, e^{-ik_1(\cdot+\theta_1)}\rangle}{2i}\\
            &=\frac{\widehat{u^2}_{\alpha,\beta}\langle E^{\alpha-\beta+(1,0)}, 1\rangle-\widehat{u^2}_{\alpha,\beta}\langle E^{\alpha-\beta-(1,0)}, 1\rangle}{2i}\\
            &\eqqcolon I_{\alpha,\beta}^+-I_{\alpha,\beta}^-.
        \end{align*}
        We see that $I_{\alpha,\beta}^+$ is nonzero only if
        \begin{align}
            (\alpha_1-\beta_1+1)k_1+(\alpha_2-\beta_2)k_2&=0.\label{eq:nm}
        \end{align}
        This equation means that we must have
        \begin{equation}\label{eq:alphabetacond+}
            \alpha_1-\beta_1+1=pk_2\quad\text{and}\quad \alpha_2-\beta_2=-pk_1\quad\text{for some }p\in\mathbb{Z}.
        \end{equation}
        Similarly we find that $I_{\alpha',\beta'}^-$ is nonzero only if
        \begin{equation}\label{eq:alphabetacond-}
            \alpha_1'-\beta_1'-1=qk_2\quad\text{and}\quad \alpha_2'-\beta_2'=-qk_1\quad\text{for some }q\in\mathbb{Z}.
        \end{equation}
        Thus $(\alpha,\beta)$ satisfies \cref{eq:alphabetacond+} if, and only if, $(\alpha',\beta')=(\beta,\alpha)$ satisfies \cref{eq:alphabetacond-} with $q=-p$. For such $\alpha,\beta$ we obtain 
        \[
            I_{\alpha,\beta}^+-I_{\beta,\alpha}^-=\frac{\widehat{u^2}_{\alpha,\beta}(e^{ipk_1k_2(\theta_1-\theta_2)}-e^{-ipk_1k_2(\theta_1-\theta_2)})}{2i}=\widehat{u^2}_{\alpha,\beta}\sin(p k_1k_2(\theta_1-\theta_2)),
        \]
        using the fact that $\widehat{u^2}_{\alpha,\beta}=\widehat{u^2}_{\beta,\alpha}$.
        Thus 
        \begin{align*}
            \langle I(r,\theta,c,\kappa,T),v_{\sin}^1\rangle &=\sum_{\vert\alpha\vert+\vert\beta\vert\geq 2} r^{\alpha+\beta}(I^+_{\alpha,\beta}-I^-_{\alpha,\beta})\\
            &=\sum_{\vert\alpha\vert+\vert\beta\vert\geq 2} r^{\alpha+\beta}(I^+_{\alpha,\beta}-I^-_{\beta,\alpha})\\
            &=\sum_{p\in\mathbb{Z}\setminus\{0\}}\sum_{\substack{\vert\alpha\vert+\vert\beta\vert\geq 2\\
                \alpha-\beta+(1,0)=p(k_2,-k_1)}}r^{\alpha+\beta}\widehat{u^2}_{\alpha,\beta}\sin(p k_1k_2(\theta_1-\theta_2)).
        \end{align*}
    In the series above we have $\alpha_1+\beta_1\geq k_2-1$ and $\alpha_2+\beta_2\geq k_1$, with equality precisely when $p=1$, $\alpha=(k_2-1,0)$, and $\beta=(0,k_1)$. Moreover, we note that for any $p\in\Z\setminus\{0\}$ there is an analytic (and bounded) function $f_p(\theta)$ such that $\sin(pk_1k_2(\theta_1-\theta_2))=\sin(k_1k_2(\theta_1-\theta_2))f_p(\theta)$,
    whence $\Psi_3$ is given by
    \begin{equation}\label{eq:Phiexplicit}
        \Psi_3(r,\theta,c,\kappa,T)=\sum_{p\in\mathbb{Z}\setminus\{0\}}\sum_{\substack{\vert\alpha\vert+\vert\beta\vert\geq 2\\
                \alpha-\beta+(1,0)=p(k_2,-k_1)}}r^{\alpha+\beta-(k_2-1,k_1)}\widehat{u^2}_{\alpha,\beta}f_p(\theta).
    \end{equation}
    This expression for $\Psi_3$ is a more precise version of \cref{eq:Psi3rexp}.
    The proof for $\langle I,v_{\sin}^2\rangle$ and $\Psi_4$ is completely analogous, and the linear dependence of $\Psi_3$ and $\Psi_4$ follows from \eqref{eq: theLinearDependencyOfTheTwoSinEquationsWrittenOutForTheFiniteDimensionalProblem}.
    \end{proof}
    From this proposition we see that $\langle I,v_{\sin}^1\rangle=\langle I,v_{\sin}^2\rangle=0$ is trivially satisfied if $r_1=0$, $r_2=0$ or $\theta_1-\theta_2\in \frac{\pi}{k_1k_2}\Z$. This is expected because these are exactly the $r$ and $\theta$ values for which $v(r,\theta)$ is symmetric and the problem reduces to the one studied in \cite{Ehrnstroem2019}. In the case when $v$ is asymmetric $\Psi_3=0$ if and only if $\Psi_4=0$ so if we show that either of these equations is satisfied it follows that $\langle I,v_{\sin}^1\rangle=\langle I,v_{\sin}^2\rangle=0$. We make the choice to solve $\Psi_3=0$. 
    
    To conclude, we will solve $(\Psi_1,\Psi_2,\Psi_3)=0$ which by the preceding factorization propositions imply that \cref{eq:lsfindim} is satisfied. This is done in two steps. First we find $c,\kappa$ such that $(\Psi_1,\Psi_2)=0$ for any $T\in (0,1/3)$, then with these solutions in hand we find $T$ such that $\Psi_3=0$.

	\begin{proposition}[Solving for $c$ and $\kappa$]\label{prop:findimpart1}
		Let the integers $2\leq k_1<k_2$ be coprime. Then for every $ T\in(0,1/3) $, $\vert r\vert\ll 1$ and $\theta\in \Theta$, there exist unique $ c=c(r,\theta,T) $ and $ \kappa=\kappa(r,\theta,T) $ satisfying
		\[
			\Psi_1(r,\theta,c,\kappa,T)=\Psi_2(r,\theta,c,\kappa,T)=0.
		\]
		Moreover, $ c $ and $ \kappa $ depends analytically on their arguments and
		\begin{equation}\label{eq:ckapparexp}
			c(0,\theta,T)=c_0+\mathcal{O}(\vert r\vert^2),\qquad \kappa(0,\theta,T)=\kappa_0+\mathcal{O}(\vert r\vert^2).
		\end{equation}
	\end{proposition}
\begin{proof}
From \Cref{proposition:Psifactorization} we see that $\Psi_1(0,\theta,c,\kappa,T)=m_{T}(\kappa k_1)-c$ and $\Psi_2(0,\theta,c,\kappa,T)=m_{T}(\kappa k_2)-c$, so
	clearly
	$ \Psi_1(0,\theta,c_0,\kappa_0,T)=\Psi_2(0,\theta,c_0,\kappa_0,T)=0 $.
	By explicit calculation we obtain
	\begin{align*}
		D_{c,\kappa}\begin{pmatrix} \Psi_1(0,\theta,c_0,\kappa_0,T)\\
		\Psi_2(0,\theta,c_0,\kappa_0,T)
		\end{pmatrix}
		&=
		\begin{pmatrix}
			D_c\Psi_1(0,\theta,c_0,\kappa_0,T) & D_\kappa\Psi_1(0,\theta,c_0,\kappa_0,T)\\
			D_c\Psi_1(0,\theta,c_0,\kappa_0,T) & D_\kappa\Psi_1(0,\theta,c_0,\kappa_0,T)
		\end{pmatrix}\\
		&=\begin{pmatrix} -1 & m'_T(\kappa_0k_1)k_1\\
		 -1 & m'_T(\kappa_0k_2)k_2
		\end{pmatrix}
	\end{align*}
		where
		\[ 
			\det\begin{pmatrix} -1 & m'_T(\kappa_0k_1)k_1\\
			-1 & m'_T(\kappa_0k_2)k_2
			\end{pmatrix}=m'_T(\kappa_0k_1)k_1-m'_T(\kappa_0k_2)k_2\neq 0,
		\]
		because $ m'_T(\kappa_0k_1) $ and $ m'_T(\kappa_0k_2) $ are both nonzero and of opposite sign by \Cref{prop: c0AndKappa0AreDeterminedByK1K2AndT}. Using the analytic implicit function theorem \cite[Theorem 4.5.4]{Buffoni_2003} we can find the desired $ c $ and $ \kappa $ and the property \cref{eq:ckapparexp} follows from \cref{eq:Psi1rexp,eq:Psi2rexp}.
	\end{proof}
Next, we wish to solve $\Psi_3=0$; for this, we make the following definition:
    \begin{definition}[Definition of $\Phi$ and $\phi$]\label{def:phidef}
         Let $c(r,\theta,T)$ and $\kappa(r,\theta,T)$ be the solutions from \Cref{prop:findimpart1} and $\Psi_3$ the function from \Cref{proposition:lastfineq}. We define the function
         \[
         \Phi(r,\theta,T)\coloneqq \Psi_3(r,\theta,c(r,\theta,T),\kappa(r,\theta,T),T),
         \]
         as well as
        \begin{equation}\label{eq:phidef}
        \phi(T;k_1,k_2)\coloneqq\Phi(0,\theta,T)= \widehat{u^2}_{(k_2-1,0),(0,k_1)}(c_0,\kappa_0,T).
    \end{equation}
    \end{definition}
    \begin{remark}
	Observe that $\phi$ depends \textit{only} on $T$ and the wave number pair $(k_1,k_2)$. This is because, by \Cref{prop: c0AndKappa0AreDeterminedByK1K2AndT}, both $c_0=c_0(T;k_1,k_2)$ and $\kappa_0=\kappa_0(T;k_1,k_2)$ are completely determined by $(k_1,k_2)$ and $T\in(0,\tfrac{1}{3})$.
	\end{remark}
    \begin{remark}\label{rem:phik1def} In the case of $k_1=1$, we can assume a solution $c(r,\theta,T)$ and $\kappa(r,\theta,T)$ to $\langle I, v^1_{\cos}\rangle=\langle I, v^2_{\cos}\rangle=0$ (see \cite{Ehrnstroem2019} for exactly when these solutions indeed exist) such that $c(0,\theta,T)=c_0$ and $\kappa(0,\theta,T)=\kappa_0$. Then $\Phi$ and $\phi$ can be defined in the same manner, which will allow us to rule out asymmetric waves in the case $k_1=1$.
    \end{remark}
    The first step to solve the final equation is to find a $T_0$ such that $\phi(T_0;k_1,k_2)=0$. If we cannot find such a $T_0$, then clearly $\Phi\neq 0$ for all $|r|\ll 1$. The next step is to extend this to some $(r,\theta)\mapsto T(r,\theta)$ such that $\Phi(r,\theta,T(r,\theta))=0$ for all $\vert r\vert \ll 1$ and $\theta$. Both of these steps can be resolved in the vicinity of a `good' root of $T\mapsto \phi(T;k_1,k_2)$ as described in \Cref{def: symmetryBreakingBifurcationPoints}. Under these good conditions, we have the following result:
	\begin{lemma}[Solving for $T$]\label{lemma:asymmetricsol} Suppose the wave number pair $2\leq k_1<k_2$ is coprime and  admits symmetry breaking in accordance with \Cref{def: symmetryBreakingBifurcationPoints} and let $(c_0,\kappa_0,T_0)$ be a corresponding symmetry breaking bifurcation point.  Then, for any $\vert r\vert \ll 1$ and $ \theta\in \Theta $ there exists $T= T(r,\theta) $ 
		\[ 
			\Phi(r,\theta,T(r,\theta))=0,
		\]
    with $T(0,\theta)=T_0$. Moreover, if the symmetry breaking bifurcation point $(c_0,\kappa_0,T_0)$ is nondegenerate, then $T$ depends analytically on its arguments and
    \begin{equation}\label{eq:Trexp}
    T=T_0+\mathcal{O}( r^2).
    \end{equation}
	\end{lemma}
	\begin{proof} By the assumptions we have a number $ T_1 $ such that $ \Phi(0,\theta,T_1)=\phi(T_1;k_1,k_2)<0 $ and a number $T_2$ such that $\Phi(0,\theta,T_2)=\phi(T_2;k_1,k_2)>0$; see \Cref{rem:phiT1T2oppositesign}. By continuity there exists some $\varepsilon>0$ such that $\Phi(r,\theta,T_1)<0$ and $\Phi(r,\theta,T_2)>0$ for all $\vert r \vert <\varepsilon$. Thus by the intermediate value theorem there exists some $T(r,\theta)$ solving $\Phi(r,\theta,T(r,\theta))=0$ for all such $r$.
	Moreover, if the symmetry breaking bifurcation point is nondegenerate, then we can directly apply the analytic implicit function theorem \cite[Theorem 4.5.4]{Buffoni_2003} to obtain analyticity of $T$ with \cref{eq:Trexp} following from \cref{eq:Psi3rexp,eq:ckapparexp}.
	\end{proof}

\subsection{Existence theorem} We are now at a point where we can prove the central existence theorem of this paper. In this theorem we will assume that we have a wave number pair that admits symmetry breaking. This technical detail is then treated separately in \Cref{sec: symmetryBreakingWavePairs} for completeness.
	\begin{theorem}\label{thm:main}
		Assume the wave number pair $ (k_1,k_2) $ satisfies $2\leq k_1< k_2$ and admits symmetry breaking. Then, for some $\varepsilon>0$, there is a mapping from $[0,\varepsilon)^2\times \Theta$ to $U^s\times\R^2_+\times (0,\tfrac{1}{3})$,
    \[
    (r,\theta)\mapsto  (u(r,\theta),c(r,\theta),\kappa(r,\theta),T(r,\theta)), 
    \]
such that $(u(r,\theta),c(r,\theta),\kappa(r,\theta),T(r,\theta))$ is a solution to \cref{eq: scaledSteadyCapillaryWhithamEquation}. Additionally, $x\mapsto u(x;r,\theta)$ is an asymmetric function if, and only if, $r_1r_2\neq 0$ and $\theta_1-\theta_2\notin \frac{\pi}{k_1k_2}\mathbb{Z}$.
		
        Moreover, if the symmetry breaking bifurcation point is nondegenerate then the solution depends analytically on $(r,\theta)$ and
        \begin{align}
        u&=r_1\cos(k_1(x+\theta_1))+r_2\cos(k_2(x+\theta_2))+\mathcal{O}( r^2),\\
        c&=c_0+\mathcal{O}( r^2),\qquad \kappa=\kappa_0+\mathcal{O}(\vert r\vert^2),\qquad T=T_0+\mathcal{O}( r^2).\label{eq:ckappaTrexp}
        \end{align}
	\end{theorem}
 
\begin{remark}\label{remark: aDescriptionOfAllSymmetricSolutions}
The solutions provided by this theorem are the ones obtained from solving $(\Psi_1,\Psi_2,\Psi_3)=0$ which, in the nondegenerate case, is uniquely solved in a neighbourhood of the bifurcation point. By the two factorization results, \Cref{proposition:Psifactorization,proposition:lastfineq}, this system is equivalent to the original \eqref{eq: finalSystem} in the asymmetric case where $r_1r_2\neq 0$ and $\theta_1-\theta_2\notin \frac{\pi}{k_1k_2}\mathbb{Z}$ (\Cref{proposition: criterionForAsymmetryOfV}), but this is no longer true when seeking symmetric waves: 

  For each of the three cases $r_1=0$, $r_2=0$, and $\theta_1-\theta_2\in\frac{\pi}{k_1k_2}\Z$, some of the equations in \eqref{eq: finalSystem} are trivially satisfied (by said factorization results) which frees up some of the bifurcation parameters. By then making appropriate changes to the above argument we end up with three additional four-dimensional manifolds, consisting of symmetric solutions of \eqref{eq: scaledSteadyCapillaryWhithamEquation}, that intersects the one from \Cref{thm:main} at $r_1=0$, $r_2=0$, and $\theta_1-\theta_2\in\frac{\pi}{k_1k_2}\Z$ respectively. These three manifolds are also characterized in \cite[Theorem 4.1]{Ehrnstroem2019} (where they are two-dimensional as neither $x$-translations nor the surface tension are included as bifurcation parameters). Thus, our result extends this picture by demonstrating the occasional presence of a fourth manifold consisting (mostly) of asymmetric solutions.
	\end{remark}

\begin{proof} As noted at the beginning of this section, \cref{eq: scaledSteadyCapillaryWhithamEquation} is equivalent to \cref{eq:lsinfdim,eq:lsfindim}. For any $\vert r\vert \ll 1$, $\theta\in\Theta$, $\vert (c-c_0,\kappa-\kappa_0)\vert \ll 1$ and $T\in (0,1/3)$ we can find solution $w=w(v(r,\theta),c,\kappa,T)$ to \cref{eq:lsinfdim} by \Cref{prop: infdimAKAuniqueWSolution}. Now it remains to solve \cref{eq:lsfindim} with the solution to \cref{eq:lsinfdim} substituted into the equation. By \Cref{corollary: JOfUIsOrthogonalToVPrime,proposition:Psifactorization,proposition:lastfineq} it is sufficient to solve $(\Psi_1,\Psi_2,\Psi_3)=0$. We solve these remaining equations in two steps to avoid requiring the assumption that the symmetry breaking bifurcation point is nondegenerate in this theorem. It is clear that we can find a solution $(c,\kappa)=(c(r,\theta,T),\kappa(r,\theta,T))$ to $(\Psi_1,\Psi_2)=0$ by \Cref{prop:findimpart1} for every $\vert r\vert \ll 1$, $\theta\in\Theta$ and $T\in(0,1/3)$. With this solution we can define $\Phi$ and find a solution $T(r,\theta)$ to  $\Phi=0$ by \Cref{lemma:asymmetricsol} for any $\vert r\vert \ll 1$ and $\theta\in\Theta$. Thus we obtain a solutions to \cref{eq: scaledSteadyCapillaryWhithamEquation} for any $\vert r\vert \ll 1$ and $\theta\in\Theta$. Now $v=P_V u$ is asymmetric by \Cref{proposition: criterionForAsymmetryOfV} if and only if $r_1\neq 0$, $r_2\neq 0$ and $\theta_1-\theta_2\notin \frac{\pi}{k_1k_2}\mathbb{Z}$ hence so is $u$ by \Cref{prop: UIsSymmetricIfAndOnlyIfVIs}.

    In the nondegenerate case a more direct proof is possible. Since
    \[
    \Psi_1(0,\theta,c_0,\kappa_0,T_0)=\Psi_2(0,\theta,c_0,\kappa_0,T_0)=\Psi_3(0,\theta,c_0,\kappa_0,T_0)=0
    \]
    and
    \[
    \det\left(\frac{\partial(\Psi_1,\Psi_2,\Psi_3)}{\partial(c,\kappa,T)}(c_0,\kappa_0,T_0)\right)=\partial_T\phi(T_0;k_1,k_2)\neq0,
    \]
    we can apply the analytic implicit function theorem \cite[Theorem 4.5.4]{Buffoni_2003} to $(\Psi_1,\Psi_2,\Psi_3)$ immediately and obtain the desired result with \cref{eq:ckappaTrexp} following from \cref{eq:Psi1rexp,eq:Psi2rexp,eq:Psi3rexp}.
	\end{proof}

\section{On wave number pairs that admit symmetry breaking}\label{sec: symmetryBreakingWavePairs}

        This section is concerned with the final piece of the puzzle, namely, if any wave number pair $(k_1,k_2)$ in fact admits symmetry breaking. This is resolved explicitly for $(2,5)$, which is (as we shall prove) the smallest pair that admits symmetry breaking. The heavily computational method used can be implemented numerically and some results in this direction are also presented towards the end of the section. In particular, numerical evidence suggests there are infinitely many wave number pairs that admit symmetry breaking.
        
        By using \cref{eq:phidef} and the recursive relations for $\hat{u}_{\alpha,\beta}$ and $\widehat{u^2}_{\alpha,\beta}$ in \cref{eq:urecursive2,eq:u2recursive} we can, in theory, compute $\phi$ for any $(k_1,k_2)$. However, the number of terms grows very fast as $k_1,k_2$ increases (see $N$ from Remark \ref{rem: ExactExpressionForNAndM}) and the $T$ dependence through $c_0$ and $\kappa_0$ is still implicit in this expression, which is why we will consider limits as $T\downarrow 0$ and $T\uparrow \frac{1}{3}$. Instead of taking the limit of $\phi$ directly, which diverges to $\pm \infty$, we will `normalize' by dividing with a positive function given by $(\ell(k_2+1))^{k_2+k_1-3}$ and consider the signs of
        \begin{equation}\label{eq:phiLims}
            \lim_{T\downarrow 0}\frac{\phi(T;k_1,k_2)}{(\ell(k_2+1))^{k_2+k_1-3}},\qquad\text{and}\qquad\lim_{T\uparrow \frac{1}{3}}\frac{\phi(T;k_1,k_2)}{(\ell(k_2+1))^{k_2+k_1-3}}.
        \end{equation}
        If the signs are different $\phi(T;k_1,k_2)$ changes sign an odd number of times on $(0,1/3)$ and thus $(k_1,k_2)$ admits symmetry breaking. While this method is computationally simpler than actually finding the $T$ dependence of $c_0$ and $\kappa_0$ it does not allow us to determine whether $(k_1,k_2)$ is symmetry breaking if $\phi$ changes sign an even number of times on $(0,1/3)$ nor can we determine whether the corresponding symmetry breaking bifurcation point is degenerate or not. To be able to compute the limits we require the following lemma.
        \begin{lemma}[Normalization for limit surface tensions]\label{lemma:l-limits} Let $\ell$ be as defined in \eqref{eq: theDefinitionOfTheMultiplierEll}. For integers $n_1,n_2\notin\{\pm k_1, \pm k_2\}$, the relative size of $\ell(n_1)$ and $\ell(n_2)$ as $T\to \{0,\frac{1}{3}\}$ is given by 
            \begin{align*}
            \lim_{T\downarrow 0}\frac{\ell(n_1)}{\ell(n_2)}&=\frac{\sqrt{k_1+k_2}-\sqrt{\frac{k_1k_2}{n_2}+n_2}}{\sqrt{k_1+k_2}-\sqrt{\frac{k_1k_2}{n_1}+n_1}}, & \lim_{T\uparrow \frac{1}{3}}\frac{\ell(n_1)}{\ell(n_2)}&=\frac{n_2^2(k_1^2+k_2^2)-k_1^2k_2^2-n_2^4}{n_1^2(k_1^2+k_2^2)-k_1^2k_2^2-n_1^4}.
            \end{align*}
        \end{lemma}
\begin{proof}
    This follows from inserting the asymptotic expressions for $\kappa_0$ and $c_0$ from \Cref{prop: c0AndKappa0AreDeterminedByK1K2AndT} into the definition of $\ell$.
\end{proof}
		\begin{proposition}\label{prop:intersection25}
			The wave number pair $(k_1,k_2)=(2,5)$ admits symmetry breaking.
		\end{proposition}
		\begin{proof}
            Most of this proof is a lengthy, but elementary computation. For brevity we do not include the entire computation, but rather the idea and the result.
			The definition of $\phi$ in \cref{eq:phidef} shows that we need to compute $\widehat{u^2}_{(4,0),(0,2)}(c_0,\kappa_0,T)$. By \cref{eq:u2recursive} we find
            \begin{align*}
                \widehat{u^2}_{(4,0),(0,2)}&=2\hat{u}_{(3,0),(0,2)}\hat{u}_{(1,0),(0,0)}
                +2\hat{u}_{(2,0),(0,2)}\hat{u}_{(2,0),(0,0)}
                +2\hat{u}_{(1,0),(0,2)}\hat{u}_{(3,0),(0,0)}\\
                &\qquad+2\hat{u}_{(0,0),(0,2)}\hat{u}_{(4,0),(0,0)}
                +2\hat{u}_{(4,0),(0,1)}\hat{u}_{(0,0),(0,1)}
                +2\hat{u}_{(3,0),(0,1)}\hat{u}_{(1,0),(0,1)}\\
                &\qquad+\hat{u}_{(2,0),(0,1)}\hat{u}_{(2,0),(0,1)}.
            \end{align*}
            The terms in this expression can be determined recursively using \cref{eq:urecursive2} beginning with
            \[
                \hat{u}_{(1,0),(0,0)}=\hat{u}_{(0,0),(0,1)}=\frac{1}{2}.
            \]
            In the next step we can determine
            \begin{align*}
            \hat{u}_{(2,0),(0,0)}&=\ell(4)\hat{u}_{(1,0),(0,0)}\hat{u}_{(1,0),(0,0)}=\frac{\ell(4)}{2^2},\\
            \hat{u}_{(1,0),(0,1)}&=2\ell(-3)\hat{u}_{(1,0),(0,0)}\hat{u}_{(0,0),(0,1)}=\frac{2\ell(3)}{2^2},\\
            \intertext{and}
            \hat{u}_{(0,0),(0,2)}&=\ell(-10)\hat{u}_{(0,0),(0,1)}\hat{u}_{(0,0),(0,1)}=\frac{\ell(10)}{2^2}.
            \end{align*}
            Continuing in this manner we end up with the explicit expression
            \begin{align*}
             \phi(T;2,5)=\frac{1}{2^6}\big[&{40}{\ell(4)\ell(6)\ell(8)\ell(10)}
                +{10}{\ell(4)\ell(4)\ell(8)\ell(10)}+{20}
                {\ell(4)\ell(4)\ell(6)\ell(10)}\\
                &+{80}{\ell(3)\ell(4)\ell(6)\ell(8)}+{20}{\ell(3)\ell(4)\ell(4)\ell(8)}
                +{40}{\ell(3)\ell(3)\ell(4)\ell(6)}\\
                &+{40}{\ell(1)\ell(4)\ell(4)\ell(6)}
                +{80}{\ell(1)\ell(3)\ell(4)\ell(6)}
                +{40}{\ell(1)\ell(3)\ell(4)\ell(4)}\\
                &+{80}{\ell(1)\ell(3)\ell(3)\ell(4)}+{20}{\ell(1)\ell(1)\ell(4)\ell(4)}
                +{80}{\ell(1)\ell(1)\ell(3)\ell(4)}\\
                &+{80}{\ell(1)\ell(1)\ell(3)\ell(3)}\big].
            \end{align*}
By appealing to \Cref{lemma:l-limits} one gets
\begin{equation}\label{eq:hatu2limits}
    \begin{aligned}
        \lim_{T\downarrow 0}\frac{\phi(T;2,5)}{(\ell(6))^4}<0\quad \text{and}\quad
        \lim_{T\uparrow \frac{1}{3}}\frac{\phi(T;2,5)}{(\ell(6))^4}>0,
    \end{aligned}
\end{equation}
which proves the result.
		\end{proof}
  Naturally, this method of explicit calculation is unusable for large $k_1,k_2$, because the number of terms eventually become to many to compute. In fact, in general we have the following.
\begin{lemma}\label{lemma:phi_explicit}
    Let $1\leq k_1<k_2$ be coprime. Then there exist corresponding numbers $N$ and $M$ and a set of integers $\{k_{i,j}\}_{\substack{1\leq i\leq M\\
    1\leq j\leq N}}$ such that 
    \[
    1\leq k_{i,j}\leq k_1k_2,\qquad k_{i,j}\neq k_1,\qquad k_{i,j}\neq k_2,
    \]
    and 
    \[
    \widehat{u^2}_{(k_2-1,0),(0,k_1)}=\frac{1}{2^{k_1+k_2-1}}\sum_{j=1}^{N}\prod_{i=1}^{M}\ell(k_{i,j}).
    \]
\end{lemma}
\begin{proof}
We begin by proving that for any positive integers $a$, $b$ there exists a corresponding $A\geq 1$, $k_{i,j}$ and $m$, $n$ depending on $i,j$ such that
\begin{equation}\label{eq:kijcondition}
k_{i,j}=\vert mk_1-nk_2\vert, \qquad m\leq a,\, n\leq b\text{ and } 2\leq m+n,
\end{equation}
and
\begin{equation}\label{eq:explicit_u_coeff}
\hat{u}_{(a,0),(0,b)}=\frac{1}{2^{a+b}}\sum_{j=1}^A\prod_{i=1}^{a+b-1}\ell(k_{i,j}).
\end{equation}
Here we use the standard convention that an empty product is equal to $1$, then it is clearly true that \cref{eq:explicit_u_coeff} holds for $a$ and $b$ such that $a+b=1$. We proceed by induction. Assume \cref{eq:explicit_u_coeff} holds for all $a+b\leq d$, then for $a+b=d+1$, by the recursion formula in \cref{eq:urecursive2} we obtain
\begin{align*}
\hat{u}_{(a,0),(0,b)}&=\sum_{\substack{a'+a''=a\\b'+b''=b}}\ell(ak_1-bk_2)
\left(
\frac{1}{2^{a'+b'}}\sum_{j'=1}^{A'}\prod_{i'=1}^{a'+b'-1}\ell(k'_{i',j'})
\right)
\left(
\frac{1}{2^{a''+b''}}\sum_{j''=1}^{A''}\prod_{i''=1}^{a''+b''-1}\ell(k''_{i'',j''})
\right)\\
&=\sum_{\substack{a'+a''=a\\b'+b''=b}}\frac{\ell(ak_1-bk_2)}{2^{a'+a''+b'+b''}}\sum_{j=1}^{A'A''}\left(\prod_{i=1}^{a'+a''+b'+b''-2}\ell(k_{i,j})\right)\\
&=\frac{1}{2^{a+b}}\sum_{j=1}^A\prod_{i=1}^{a+b-1}\ell(k_{i,j}),
\end{align*}
where all $k_{i,j}=k_{i',j'}'$, $k_{i,j}=k_{i'',j''}$ or $k_{i,j}=\vert ak_1-bk_2\vert $ for some an appropriate choice of indexing. in particular this means that \cref{eq:kijcondition} holds for all $k_{i,j}$. Now, since $\phi(T;k_1,k_2)=\widehat{u^2}_{(k_2-1,0),(0,k_1)}$ we can use the formula in \cref{eq:u2recursive} and the same computation as above to obtain
\[
\phi(T;k_1,k_2)=\frac{1}{2^{k_1+k_2-1}}\sum_{j=1}^{N}\prod_{i=1}^{M}\ell(k_{i,j}).
\]
Moreover, all $k_{i,j}=\vert mk_1-nk_2\vert$ for some $m\leq k_2-1$ and $n\leq k_1$ and $2\leq m+n\leq k_1+k_2-2$. Combined with the fact that $k_1$ and $k_2$ are coprime we obtain
\[
    1\leq k_{i,j}\leq k_1k_2,\qquad k_{i,j}\neq k_1,\qquad k_{i,j}\neq k_2,
\]
which finishes the proof.
\end{proof}
\begin{remark}\label{rem: ExactExpressionForNAndM}
From the proof of the lemma it is clear that we in fact have
\[
M=k_1+k_2-3.
\]
It is also possible to compute the value of $N$, which is given by
\[
    N=\frac{(2k_2+2k_1-4)!}{(k_1+k_2-2)!k_1!(k_2-1)!}.
\]
(As this expression is not needed for the analysis, we omit its cumbersome proof.) This, combined with Sterling's formula, implies $N\simeq 64^k/k^2$ when $k_1\approx k_2\eqqcolon k\to \infty$. Thus, it quickly becomes unfeasible to compute $\phi$ manually as $k_1,k_2$ grows.
\end{remark}
    \begin{figure}[H]
                \includegraphics[scale=1.0]{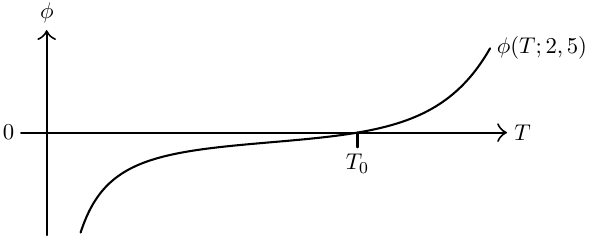}
                \caption{Plot of the function $T\mapsto\phi(T;2,5)$ on the interval $(0,1/3)$. Here $\phi(T_0;2,5)=0$ for a value $T_0\approx 0.1215$.}
                \label{fig:25Tplot}
        \end{figure}
\noindent While we lack the means to determine whether $T\mapsto\phi(T;k_1,k_2)$ crosses zero for a general wave number pair $(k_1, k_2)$, there are two groups of such pairs where symmetry breaking can be readily ruled out by showing there exists no $T_0\in(0,1/3)$ such that $\phi(T_0;k_1,k_2)=0$. This is important because the existence of such a $T_0$ is necessary for the existence of asymmetric waves.
        \begin{proposition}\label{prop:phizeronecessary}
        If $\phi(T;k_1,k_2)\neq 0$ for all $T\in (0,1/3)$ then there exists no asymmetric solutions to \cref{eq: scaledSteadyCapillaryWhithamEquation} with arbitrarily small amplitude.
        \end{proposition}
        \begin{proof}
        Assume $\phi(T;k_1,k_2)> 0$ for $T\in (0,1/3)$ (otherwise replace $\phi$ with $-\phi$ in the proof), then there exists some constant $\epsilon>0$ such that $\phi(T;k_1,k_2)> \epsilon$. By continuity there exist some $\delta$ such that $\Phi(r,\theta,T)>\epsilon/2$ on $T\in(0,1/3)$ for all $\vert r\vert <\delta$. This is uniform in $T$ even though $(0,1/3)$ is open because $\phi(T;k_1,k_2)\to \infty$ as $T\to 0$ or $T\to 1/3$. Since $\langle I, v_{\sin}^1\rangle=r_1^{k_1-1}r_2^{k_1}\sin(k_1k_2(\theta_1-\theta_2))\Phi$ (for $c$, $\kappa$ such that $\langle I, v_{\cos}^1\rangle=\langle I, v_{\cos}^2\rangle=0)$, we get $\langle I, v_{\sin}^1\rangle\neq 0$ for $r_1^{k_1-1}r_2^{k_1}\sin(k_1k_2(\theta_1-\theta_2))\neq 0$ which is necessary for an asymmetric $v$.
        \end{proof}
        Now we can show two special cases when $\phi(T;k_1,k_2)$ is nonzero.
		\begin{proposition}\label{prop:NoSymmetryBreaking}
			If $ k_1=1 $ or $k_2- k_1=1 $, then the pair $ (k_1,k_2) $ cannot admit symmetry breaking in accordance with Definition \ref{def: symmetryBreakingBifurcationPoints}. In fact, $\phi(T;1,k_2)\neq 0$ and $\phi(T;k_2-1,k_2)\neq 0$ for all $T\in(0,1/3)$.
		\end{proposition}
		\begin{proof}
			From \Cref{def:phidef} and \Cref{lemma:phi_explicit} (for $k_1=1$ see \Cref{rem:phik1def} as well) we get that $\phi(T;1,k_2)$ is of the form
            \[
               \phi(T;1,k_2)=\frac{1}{2^{k_2}}\sum_{j=1}^N \prod_{i=1}^{M}\ell(k_{i,j})
            \]
            where $1<k_{i,j}<k_2$. In particular, for all such $k_{i,j}$ and $T\in (0,1/3)$ we have $\ell(k_{i,j})>0$, thus $\phi(T;1,k_2)> 0$ for all $T\in(0,1/3)$.

            On the other hand, $\phi(T;k_2-1,k_2)$ is of the form
            \[
            \phi(T;k_2-1,k_2)=\frac{1}{2^{2k_2-2}}\sum_{j=1}^N \prod_{i=1}^{M}\ell(k_{i,j})
            \]
            where $k_{i,j}<k_2-1=k_1$ or $k_2<k_{i,j}$. For all such $k_{i,j}$ and $T\in (0,1/3)$ we have $\ell(k_{i,j})<0$. Thus, it follows that $(-1)^M\phi(T;k_2-1,k_2)> 0$ for all $T\in(0,1/3)$.
		\end{proof}
		By the two previous propositions, there are no small asymmetrical solutions near double bifurcation points where either $k_1=1$ or $k_2-k_1=1$. The scaling argument from Remark \ref{rem: oneMayAssumeThatK1AndK2AreRelativePrimes} ensures that this result carries over to general (not necessary coprime) integer pairs $1\leq k_1<k_2$; the corresponding conditions become ``$k_1$ divides $k_2$'' and ``$k_2-k_1$ divides $k_1$'' and so we have:
  \begin{corollary}\label{corollary:nores}
			There are no arbitrarily small solutions to \cref{eq: scaledSteadyCapillaryWhithamEquation} such that $ u $ is asymmetric for the resonant case $ k_1 \vert k_2 $, nor are there such solutions for the case $(k_2-k_1)\vert k_1 $.
		\end{corollary}
        \begin{remark} \Cref{prop:findimpart1} can be extended to include $k_1=1$ to give solutions to $(\Psi_1,\Psi_2)=0$ for all $r$ in a slit disc, as was done in \cite{Ehrnstroem2019}. However, \Cref{corollary:nores} shows that in this case there exists no small asymmetric solutions, which is why we omit that case in this paper.
        \end{remark}
		\subsection{Numerical results}
        It is possible to numerically compute more pairs than $(k_1,k_2)=(2,5)$ that admits symmetry breaking using the same scheme as in \Cref{prop:intersection25}. The results are presented in \Cref{fig:k1k2pairs}. Of course, this scheme yields no conclusion if the limits in \cref{eq:phiLims} have the same sign; in this case it is not clear whether $\phi$ crosses zero or not. In particular, there may be several such symmetry breaking pairs $(k_1,k_2)$ missing from Figure \ref{fig:k1k2pairs}. However, from \cref{prop:NoSymmetryBreaking} we have the inequalities 
        \[
        \frac{k_2}{k_2-1}<\frac{k_2}{k_1}<k_2,
        \]
        but the results in the figure indicate these inequalities are far from sharp for large $k_2$. For example, $2<\frac{k_2}{k_1}$ for all the plotted pairs $(k_1,k_2)$, while clearly $\frac{k_2}{k_2-1}\approx 1$ for large $k_2$.
        
        We can also solve for $(c_0,\kappa_0)$ numerically and plot the $\phi(T;k_1,k_2)$ for a given pair $(k_1,k_2)$. This is done for $\phi(T;2,5)$ in \Cref{fig:25Tplot}. Doing this allows us not only to see that indeed $(2,5)$ admits symmetry breaking, but that it should be nondegenerate, something we expect to be true in general for wave pairs that admit symmetry breaking. We can also find an approximate value for $T_0$ such that $\phi(T_0;2,5)=0$.
        
    \begin{figure}[h!]
                \includegraphics[scale=0.9]{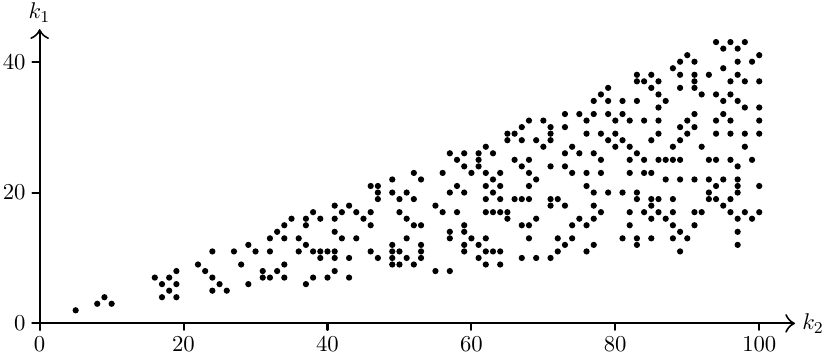}
                \caption{The dots mark coprime pairs $1\leq k_1<k_2$ for which we obtained different signs in the limits \cref{eq:phiLims} by a numerical computation. Thus all plotted wave number pairs admit symmetry breaking.}
                \label{fig:k1k2pairs}
            \end{figure}
\subsection{Other nonlinearities} Without going into the technical details we conclude with mentioning that the method of finding asymmetric waves presented in this paper generalizes quite readily to a larger class of analytic nonlinearities. If we replace \cref{eq: scaledSteadyCapillaryWhithamEquation} with
\[
(M_{T,\kappa}-c)u+\mathcal{N}(u)=0,
\]
where
\begin{equation}\label{eq:newnonlin}
\mathcal{N}(u)=\sum_{n=2}^{\infty} a_n u^n,\qquad a_2\neq 0,
\end{equation}
then solving the corresponding $\Psi_1=\Psi_2=0$ can be done in precisely the same manner because this only relies on the linear part of the equation. Solving $\Psi_3=0$ is more involved than with a square nonlinearity, but it is fairly straightforward to find that
\[
\phi^{\mathcal{N}}(T;k_1,k_2)=a_2^{k_1+k_2-2}\phi(T;k_1,k_2)+R,
\]
if we let $\phi^{\mathcal{N}}(T;k_1,k_2)$ be the $\phi$-function for the nonlinearity $\mathcal{N}$ and keep $\phi(T;k_1,k_2)$ as the $\phi$-function for the square nonlinearity in \cref{eq: scaledSteadyCapillaryWhithamEquation}. Moreover, $R$ is of the form
\[
R=\sum_{i=1}^{N_R} C_i\prod_{j=1}^{M_i}\ell(k_{i,j}^R),
\]
where $M_i<M=k_1+k_2-3$. It follows that
\begin{align*}
\lim_{T\downarrow 0}\frac{\phi^{\mathcal{N}}(T;k_1,k_2)}{(\ell(k_2+1))^{k_2+k_1-3}}=a_2^{k_1+k_2-2}\lim_{T\downarrow 0}\frac{\phi(T;k_1,k_2)}{(\ell(k_2+1))^{k_2+k_1-3}},\\
\lim_{T\uparrow \frac{1}{3}}\frac{\phi^{\mathcal{N}}(T;k_1,k_2)}{(\ell(k_2+1))^{k_2+k_1-3}}=a_2^{k_1+k_2-2}\lim_{T\uparrow \frac{1}{3}}\frac{\phi(T;k_1,k_2)}{(\ell(k_2+1))^{k_2+k_1-3}}.
\end{align*}
Whence our proof that $(k_1,k_2)=(2,5)$ admits symmetry breaking for a quadratic nonlinearity carries over to a nonlinearity of the form given in \cref{eq:newnonlin}.

\section*{Acknowledgements}
O. Mæhlen was supported by the Research
Council of Norway project INICE, project no. 301538. D. S. Seth was supported by an ERCIM `Alain Bensoussan’ Fellowship Programme, the Research Council of Norway under Grant Agreement No. 325114, and by the Swedish Research Council under grant no. 2021-06594 while in residence at Institut Mittag-Leffler in Djursholm, Sweden during the fall semester of 2023. The Authors would also like to thank Erik Wahlén for a suggestion that significantly simplified the proof of \Cref{corollary: JOfUIsOrthogonalToVPrime}.
\bibliographystyle{siam}
\bibliography{RefsAsymmetric.bib}
\,\\
\end{document}